\documentclass[12pt]{amsart}
\usepackage{amssymb,times}

\makeatletter
\@namedef{subjclassname@2010}{%
  \textup{2010} Mathematics Subject Classification}
\makeatother

\parskip=\medskipamount

\newtheorem{Thm}{Theorem}
\newtheorem{theorem}[Thm]{Theorem}
\newtheorem{Cor}[Thm]{Corollary}
\newtheorem{Prop}[Thm]{Proposition}
\newtheorem{Lemma}[Thm]{Lemma}
\newtheorem{lemma}[Thm]{Lemma}

\theoremstyle{definition}
\newtheorem{Defn}{Definition}
\newtheorem{mydef}[Defn]{Definition}
\newtheorem{Remark}{Remark}

\newcommand{\mf}[1]{\mathbb{#1}}
\newcommand{\mc}[1]{\mathcal{#1}}
\newcommand{\mb}[1]{\mathbf{#1}}

\DeclareMathOperator{\NC}{\mathrm{NC}}
\DeclareMathOperator{\Part}{\mathcal{P}}
\DeclareMathOperator{\Int}{\mathit{Int}}

\newcommand{\norm}[1]{\left\Vert#1\right\Vert}
\newcommand{\abs}[1]{\left\vert#1\right\vert}
\newcommand{\chf}[1]{\mathbf{1}_{#1}}
\newcommand{\set}[1]{\left\{#1\right\}}
\renewcommand{\phi}{\varphi}
\newcommand{\eps}{\varepsilon}

\allowdisplaybreaks[1]

\begin{document}

\baselineskip=17pt

\title{Higher variations for free L{\'e}vy processes}
\author{Michael Anshelevich, Zhichao Wang}
\thanks{This work was supported in part by a Simons Foundation Collaboration Grant}
\address{Department of Mathematics, Texas A\&M University, College Station, TX 77843-3368}
\thanks{Parts of this article form part of the second author's 2018 Master's thesis at Texas A\&M University}
\email{manshel@math.tamu.edu, wangzc@tamu.edu}
\date{\today}

\begin{abstract}
For a general free L{\'e}vy process, we prove the existence of its higher variation processes as limits in distribution, and identify the limits in terms of the L{\'e}vy-It{\^o} representation of the original process. For a general free compound Poisson process, this convergence holds in probability. This implies joint convergence in distribution to a $k$-tuple of higher variation processes, and so the existence of $k$-fold stochastic integrals as limits in probability. If the existence of moments of all orders is assumed, the result holds for free additive (not necessarily stationary) processes and more general approximants. In the appendix we note relevant properties of symmetric polynomials in non-commuting variables.
\end{abstract}

\maketitle

\section{Introduction}

A free (additive) L{\'e}vy process (in law; we will typically omit this qualifier) is a family of self-adjoint random variables $\set{X(t): t \geqslant 0}$ affiliated to a non-commutative probability space $(\mc{A}, \tau)$ which starts at zero, has free, stationary increments, and is stochastically continuous:
\begin{enumerate}
\item
$X(0) = 0$,
\item
For all $n\in\mf{N}$ and $t_0 < t_1 < \ldots < t_n$,
\[
\set{X(t_0), X(t_1) - X(t_0), \ldots, X(t_n) - X(t_{n-1})}
\]
are free,
\item
The distribution of the increment $X(t + h) - X(t)$ depends only on $h$ (and will be denoted $\mu_h$),
\item
For all $\eps > 0$, $\lim_{h \rightarrow 0} \mu_{h}(\abs{x} > \eps) = 0$.
\end{enumerate}

The distributions of increments of a free L{\'e}vy process form a semigroup with respect to the additive free convolution $\boxplus$, and so are $\boxplus$-infinitely divisible. This implies that the Voiculescu transform of the distribution $\mu_t$ of $X(t)$ has the form
\begin{equation}
\label{Eq:Voiculescu-triple}
\phi_{\mu_t}(z) = t \eta + t \frac{a}{z} + t \int_{\mathbb{R}} \left[ \frac{z^{2}}{z - x} - z - x \mathbf{1}_{[-1,1]}(x) \right] \,d\rho(x),
\end{equation}
where $\eta \in \mf{R}$, $a \in \mf{R}_+$, and $\rho$ is a L{\'e}vy measure. Barndorff-Nielsen and Thor\-bj{\o}rn\-sen proved that a free L{\'e}vy process has a free L{\'e}vy-It{\^o} decomposition.

\begin{Thm}[Theorems 6.4, 6.5 in \cite{BN-T-Levy-Ito}]
\label{Thm:Levy-Ito}
Let $\set{X(t): t \geqslant 0}$ be a free L{\'e}vy process, with the generating triple $(\eta,a,\rho)$ as above. Then, $X(t)$ is equal in distribution to a sum of three freely independent parts.
In general,
\begin{multline}\label{levyde}
X(t)\overset{d}{=}\eta t \mathbf{1}_{\mathcal{A}^{0}}+ \sqrt{a} S(t)
 \\
+\lim_{\epsilon\searrow 0} \Bigl(\int_{(0,t]\times \set{|x|>\epsilon}}xdM(t,x)
-\int_{(0,t]\times \set{\epsilon<|x|\leqslant 1}}x (\mathrm{Leb} \otimes \rho)(dt,dx)\mathbf{1}_{\mathcal{A}^{0}} \Bigr).
\end{multline}
In particular, when $\int_{[-1,1]}|x|\rho(dx)$ is finite and $\tilde{\eta}:=\eta-\int_{-1}^{1}x\rho(dx),$ then
\begin{equation}\label{levyde1}
X(t)\overset{d}{=}\tilde{\eta} t \mathbf{1}_{\mathcal{A}^{0}}+ \sqrt{a} S(t)
+\int_{(0,t]\times\mathbb{R}}xdM(t,x).
\end{equation}
Here, $S(t)$ is the free Brownian motion (in some $W^{\ast}$-probability space $(\mathcal{A}^{0},\tau^{0})$) and $M$ is a free Poisson random measure on the measure space $(\mf{R}_+ \times \mf{R},\mathcal{B}(\mf{R}_+ \times \mf{R}), \mathrm{Leb} \otimes \rho)$ with values in $(\mathcal{A}^{0},\tau^{0})$. The limit is taken in probability.
\end{Thm}

In the representation in the theorem above, define the \emph{$k$'th variation} of the process by $X^{(1)}(t) = X(t)$ and for $k\geqslant 2$,
\begin{equation}
\label{Eq:k-th-variation}
X^{(k)}(t) = a t \delta_{k,2}\textbf{1}_{\mathcal{A}}+\int_{(0,t]\times\mathbb{R}}x^{k}dM(t,x).
\end{equation}
We will show that these objects are well defined, and again form a free L{\'e}vy process. Later in the article we will define the corresponding object when $x^k$ is replaced by a more general function $p(x)$.

Our first main result concerns convergence in distribution to a higher variation process.

\begin{Thm}
\label{Thm:Variations-distribution}
For each $N \in \mf{N}$, let $\set{X_{i, N} : i \in \mf{N}}$ be free, identically distributed, self-adjoint random variables affiliated to $(\mc{A}, \tau)$. Suppose that for $t \geqslant 0$,
\[
\lim_{N \rightarrow \infty} \sum_{i=1}^{[N t]} X_{i, N} \stackrel{d}{=} X(t).
\]
Then for each $k$,
\[
\lim_{N \rightarrow \infty} \sum_{i=1}^{[N t]} X_{i, N}^k \stackrel{d}{=} X^{(k)}(t),
\]
the limits being taken in distribution.
\end{Thm}

We next discuss \emph{joint} convergence in distribution. In the non-commutative case, there is at this point no universally accepted definition of this notion. Recall the following.

\begin{Defn}
A family of self-adjoint operators $(a_{1, N}, \ldots, a_{k, N})$ affiliated to a non-commutative probability space $(\mc{A}, \tau)$ converges to $(a_1, \ldots, a_k)$ \emph{jointly in moments} if for any non-commutative self-adjoint polynomial $P(x_1, \ldots, x_k)$, $\tau[P(a_{1, N}, \ldots, a_{k, N})]$ is well-defined and
\[
\tau[P(a_{1, N}, \ldots, a_{k, N})] \rightarrow \tau[P(a_1, \ldots, a_k)]
\]
The family converges \emph{jointly in distribution} if for any $P$ as above,
\[
P(a_{1, N}, \ldots, a_{k, N}) \rightarrow P(a_1, \ldots, a_k)
\]
in distribution (see \cite{Mai-Speicher-Berry-Esseen} for a related notion).
\end{Defn}

Recall that convergence in distribution and convergence in moments coincide for bounded operators, but in general neither implies the other.

The next result applies to free additive processes whose increments are not necessarily stationary.

\begin{Thm}
\label{Thm:Joint-moments}
For each $N \in \mf{N}$, let $\set{X_{i, N} : i \in \mf{N}}$ be free self-adjoint random variables affiliated to $(\mc{A}, \tau)$ all of whose moments are finite. Suppose that for $t \geqslant 0$,
\[
\sum_{i=1}^{[N t]} X_{i, N}
\]
converges in moments to $X(t)$ as $N \rightarrow \infty$. Suppose in addition that
\begin{equation}
\label{Eq:root-N}
\sum_{i=1}^N \tau[X_{i, N}^k]^2 \rightarrow 0
\end{equation}
as $N \rightarrow \infty$, for all $k$. Then there exist free additive processes $\set{X^{(j)}(t)}$ such that we have joint convergence in moments
\[
\left( \sum_{i=1}^{[N t]} X_{i, N}, \sum_{i=1}^{[N t]} X_{i, N}^2, \ldots, \sum_{i=1}^{[N t]} X_{i, N}^k \right) \rightarrow \left(X(t), X^{(2)}(t), \ldots, X^{(k)}(t) \right)
\]
as $N \rightarrow \infty$.
\end{Thm}

\begin{Remark}
For triangular arrays of centered random variables with finite variance, the standard condition for convergence is $\max_{1 \leqslant i \leqslant N} \tau[X_{i,N}^2] \rightarrow 0$ and $\sum_{i=1}^N \tau[X_{i,N}^2] \leqslant c < \infty$, see for example Section~22 in \cite{Loeve-Probability-I}. The assumption \eqref{Eq:root-N} is clearly significantly stronger. On the other hand, it is significantly weaker that assuming that all $X_{i, N}$ are identically distributed. In the latter case, the result follows from the limit theorem 13.1 in \cite{Nica-Speicher-book}, itself based on a result of Speicher \cite{Spe90}.
\end{Remark}

The second case where we can prove joint convergence is when individual convergence holds in probability.

\begin{Thm}
\label{Thm:Compound-Poisson}
Let $\rho$ be a finite probability measure, and
\[
X(t) = \int_{(0, t] \times \mf{R}} x \,dM(t,x)
\]
the corresponding free compound Poisson process. Then for $X_{i, N} = X(\frac{i}{N}) - X(\frac{i-1}{N})$, we have
\[
\lim_{N \rightarrow \infty} \sum_{i=1}^{[N t]} X_{i, N}^k = X^{(k)}(t),
\]
the limit being taken in probability.
\end{Thm}

We expect similar convergence for general free L{\'e}vy processes. At this point we have the following partial result.

\begin{Thm}
\label{Thm:Quadratic-Cubic}
Let $\set{X(t): t \geqslant 0}$ be a free L{\'e}vy process whose increments have symmetric distributions, and $X_{i, N} = X(\frac{i}{N}) - X(\frac{i-1}{N})$. Then
\[
\lim_{N \rightarrow \infty} \sum_{i=1}^{[N t]} X_{i, N}^2 = X^{(2)}(t),
\]
the limit being taken in probability.
\end{Thm}

\begin{Cor}
\label{Cor:Joint-probability}
For a free compound Poisson process $\set{X(t) : t \geqslant 0}$ and increments $X_{i, N}$ as above, we have joint convergence in distribution
\[
\left( \sum_{i=1}^{[N t]} X_{i, N}, \sum_{i=1}^{[N t]} X_{i, N}^2, \ldots, \sum_{i=1}^{[N t]} X_{i, N}^k \right) \rightarrow \left( X(t), X^{(2)}(t), \ldots, X^{(k)}(t) \right)
\]
as $N \rightarrow \infty$.
\end{Cor}

\begin{Cor}
\label{Cor:Convergence-integral}
Let $\set{X_{i, N} : 1 \leqslant i \leqslant N, N \in \mf{N}}$ be as in either Theorem~\ref{Thm:Joint-moments} or in Corollary~\ref{Cor:Joint-probability}. Then for $t \geqslant 0$,
\begin{multline}
\label{Eq:Stochastic-integral}
\lim_{N \rightarrow \infty} \sum_{\substack{1 \leqslant i(1), i(2), \ldots, i(k) \leqslant [N t] \\ i(1) \neq i(2), i(2) \neq i(3), \ldots, i(k-1) \neq i(k)}} X_{i(1), N} X_{i(2), N} \ldots X_{i(k), N} \\
= \sum_{j=1}^k (-1)^{k - j} \sum_{\substack{m_1, \ldots, m_j \geqslant 1 \\ m_1 + \ldots + m_j = k}} X^{(m_1)}(t) \ldots X^{(m_j)}(t).
\end{multline}
Here under the assumptions of Theorem~\ref{Thm:Joint-moments} the limit is in moments, while under the assumptions of Corollary~\ref{Cor:Joint-probability} the limit is in probability, and so also in distribution.
\end{Cor}

It was shown in Proposition~1 of \cite{AnsFSM} that for free L{\'e}vy processes with bounded, centered increments, the limits (in norm) of the left-hand side of \eqref{Eq:Stochastic-integral} and of
\begin{equation}
\label{Eq:All-distinct}
\sum_{\substack{1 \leqslant i(1), i(2), \ldots, i(k) \leqslant [N t] \\ \abs{\set{i(1), i(2), \ldots, i(k)}} = k}} X_{i(1), N} X_{i(2), N} \ldots X_{i(k), N}.
\end{equation}
coincide. These limits should be interpreted as the free stochastic integral
\[
\int_{[0,t]^k} \,dX(s_1) \ldots \,dX(s_k).
\]
See the end of the introduction, and the appendix, for the explanation of why the expression \eqref{Eq:Stochastic-integral} is more appropriate in the free case.

\subsection*{Prior results}

The initial motivation for our analysis was the article \cite{Avram-Taqqu-symmetric-poly} by Avram and Taqqu. We briefly compare some of their results with ours; the reader should consult their article for more details. Let $\set{X(t)}$ be a L{\'e}vy process, and define its higher variations pathwise using jumps. Note that such a definition is unavailable in the non-commutative case. Representation (\ref{Eq:k-th-variation}), which we use instead, is closely related to Theorem 36 in section I of \cite{protter2004stochastic} (where it is stated only for the jumps of a L{\'e}vy process), and is (as Protter points out) obvious in the classical case. Let $\set{X_{i, N} : 1 \leqslant i \leqslant N, N \in \mf{N}}$ be a triangular array with i.i.d. rows, such that
\[
\sum_{i=1}^N X_{i,N} \rightarrow X(t)
\]
in distribution as $N \rightarrow \infty$. Then a multivariate limit theorem implies that
\begin{equation}
\label{Eq:Joint-convergence}
\sum_{i=1}^{[N t]} \left( X_{i, N}, X_{i, N}^2, \ldots, X_{i,N}^k \right) \rightarrow \left( X(t), X^{(2)}(t), \ldots, X^{(k)}(t) \right)
\end{equation}
jointly in distribution. At this point, in the non-commutative case such a theorem is only available for convergence in moments. On the other hand, we actually prove Theorem~\ref{Thm:Variations-distribution} not just for powers but for polynomials, that is, linear combinations of powers. For commuting variables, convergence in distribution of linear combinations is equivalent to joint convergence in distribution (an easy exercise left to the reader). So the appropriate commutative analog of Theorem~\ref{Thm:Variations-distribution} also implies the joint convergence in \eqref{Eq:Joint-convergence}.

Next, recall that the elementary symmetric polynomial
\[
e_k(x_1, \ldots, x_N) = \sum_{1 \leqslant i(1) < i(2) < \ldots < i(k) \leqslant N} x_{i(1)} x_{i(2)} \ldots x_{i(k)}
\]
is a polynomial $P_k(p_1, \ldots, p_k)$ in the power sum symmetric polynomials
\[
p_j(x_1, \ldots, x_N) = \sum_{i=1}^N x_i^j
\]
(the polynomial $P_k$ can be written down explicitly). Consequently,
\begin{multline*}
\sum_{1 \leqslant i(1) < i(2) < \ldots < i(k) \leqslant [N t]} X_{i(1), N} X_{i(2), N} \ldots X_{i(k), N} \\
= P_k \left( \sum_{i=1}^{[N t]} X_{i, N}, \sum_{i=1}^{[N t]} X_{i, N}^2, \ldots, \sum_{i=1}^{[N t]} X_{i, N}^k \right)
\end{multline*}
converges in distribution as $N \rightarrow \infty$. Its limit is naturally identified with the multiple integral
\[
\int_{0 \leqslant s_1 < s_2 < \ldots < s_k \leqslant t} \,dX(s_1) \,dX(s_2) \ldots \,dX(s_k).
\]
Note that as explained in the appendix, if the variables $\set{x_i}$ do not commute, $e_k$ is \emph{not} a polynomial in the $p_j$'s. Its natural replacement in the non-commutative setting is
\[
\tilde{e}_k(x_1, \ldots, x_N) = \sum_{\substack{1 \leqslant i(1), i(2), \ldots, i(k) \leqslant N \\ i(1) \neq i(2), i(2) \neq i(3), \ldots, i(k-1) \neq i(k)}} x_{i(1)} x_{i(2)} \ldots x_{i(k)}
\]
used in equation~\eqref{Eq:Stochastic-integral}.

Motivated by \cite{Rota}, the first author studied related objects in \cite{AnsFSM}, but only for the case of free L{\'e}vy processes with compactly supported distributions. We are not aware of other sources where these specific topics are studied in the free probability setting. See however the study of homogeneous sums in \cite{Deya-Nourdin-Invariance,Simone-Universality}.

The article is organized as follows. After the introduction and background in Section~\ref{Sec:Background}, Section~\ref{Sec:Single-variable} treats, for general free L{\'e}vy processes, convergence in distribution to the higher variation processes, and their generalization from powers to more general continuous functions. The key result is Theorem~\ref{Thm:Conergence-p}. Section~\ref{Sec:Moments} treats joint convergence in moments for more general additive processes. Section~\ref{Sec:Probability} contains results about convergence in probability, as well as an alternative definition of joint convergence in distribution for non-commuting variables. Finally, in the appendix we explain which symmetric polynomials in non-commuting variables can be expressed in terms of the basic power sum symmetric polynomials.

\subsection*{Acknowledgements}
The authors are grateful to Matthieu Josuat-Verg{\`e}s for the references in the Appendix, to the referee for several useful comments and questions, and to Zhiyuan Yang for comments leading to the correction of Proposition 20.

\section{Background and the free Poisson random measure}
\label{Sec:Background}

\subsection{Unbounded Operators and Affiliated Operators}

A $W^{\ast}$-probability spa\-ce is a pair ($\mathcal{A},\tau$), where $\mathcal{A}$ is a von Neumann algebra acting on a Hilbert space and $\tau$ is a faithful normal tracial state on $\mathcal{A}$. Throughout most of the paper, we will work with possibly unbounded operators affiliated to $\mc{A}$. A self-adjoint operator $a$ is affiliated to $\mc{A}$ if all of its spectral projections are in $\mc{A}$. Equivalently, for any bounded Borel function, $f(a) \in \mc{A}$. We denote the collection of all self-adjoint operators affiliated to $\mc{A}$ by $\tilde{\mc{A}}_{sa}$. A general closed, densely defined operator $a$ is affiliated to $\mc{A}$ if in its polar decomposition $a = u \abs{a}$, we have $u \in \mc{A}$ and $\abs{a} \in \tilde{\mc{A}}_{sa}$. The collection of all such operators is denoted by $\tilde{\mc{A}}$. Murray and von Neumann \cite{Murray-vN} proved that $\tilde{\mc{A}}$ is an algebra, that is, if $a, b \in \tilde{\mc{A}}$, then $a + b$ and $a b$ are densely defined and closable, and their closures are in $\tilde{\mc{A}}$.

For $a \in \tilde{\mc{A}}_{sa}$, its distribution is the unique probability measure $\mu_a$ on $\mf{R}$ such that for any bounded Borel function,
\begin{equation}
\label{Eq:Distribution}
\tau[f(a)] = \int_{\mf{R}} f(x) \,d\mu_a(x).
\end{equation}

\begin{mydef}(\cite{BNTSelf})  Let $(\mathcal{A},\tau)$ be a $W^{\ast}$-probability space and $(a_{n})_{n\in\mathbb{N}}$ be a sequence of operators affiliated with $\mathcal{A}.$ We say that $a_{n}\rightarrow a$ in probability if $|a_{n}-a|\rightarrow 0$ in distribution as $n\rightarrow\infty$.
\end{mydef}
Here, $|a|:=\sqrt{a^{*}a}$, which is self-adjoint. When $a_{n}$ and $a$ are self-adjoint operators affiliated with $\mathcal{A}$, $a_{n}\rightarrow a$ in probability if and only if $a_{n}-a$ converges to zero in distribution, i.e. the distribution of $a_{n}-a$ as a probability measure on $\mathbb{R}$ converges weakly to probability measure $\delta_{0}.$

We list the following proposition for completeness. See for example Proposition~2.18 in \cite{BNTSelf}.

\begin{Prop} \label{proposition_probability}
The following are equivalent.
\begin{enumerate}
\item
$a_n \rightarrow a$ in probability.
\item
$\forall \eps > 0$, the traces of the spectral projections $\tau[\chf{(\eps, \infty)}(\abs{a_n - a})] \rightarrow 0$.
\item
Denote
\[
\mc{N}(\eps, \delta) = \set{b \in \tilde{\mc{A}} : \exists \text{ projection } p \in \mc{A} \text{ s.t. } \tau[1 - p] < \delta, b p \in \mc{A}, \norm{b p} < \eps}.
\]
Then $\forall \eps, \delta > 0$, for sufficiently large $n$, $a_n - a \in \mc{N}(\eps, \delta)$.
\end{enumerate}
This mode of convergence is also called convergence in measure.
\end{Prop}

We also recall the following part of Theorem~1 from \cite{Nelson-NC-integration}.

\begin{Lemma}
\label{Lemma:product-probability}
In the notation of the preceding proposition,
\[
\mc{N}(\eps_1, \delta_1) + \mc{N}(\eps_2, \delta_2) \subseteq \mc{N}(\eps_1 + \eps_2, \delta_1 + \delta_2)
\]
and
\[
\mc{N}(\eps_1, \delta_1) \mc{N}(\eps_2, \delta_2) \subseteq \mc{N}(\eps_1 \eps_2, \delta_1 + \delta_2).
\]
In particular, if $a_n \rightarrow a$ and $b_n \rightarrow b$ in probability, then $a_n + b_n \rightarrow a + b$ and $a_n b_n \rightarrow a b$ in probability.
\end{Lemma}

\subsection{Freely infinitely divisible distributions and limit theorems}

As already mentioned in the introduction, a probability measure $\mu$ on $\mf{R}$ is $\boxplus$-infinitely divisible if and only if its Voiculescu transform has a representation
\begin{equation}
\label{Eq:Voiculescu-triple}
\phi_{\mu}(z) = \eta + \frac{a}{z} + \int_{\mathbb{R}} \left[ \frac{z^{2}}{z - x} - z - x \mathbf{1}_{[-1,1]}(x) \right] \,d\rho(x),
\end{equation}
where $\eta \in \mf{R}$, $a \in \mf{R}_+$, and $\rho$ is a L{\'e}vy measure, that is,
\[
\rho(\set{0}) = 0 \text{ and } \int_{\mf{R}} \min(1, x^2) \,d\rho(x) < \infty.
\]
$\phi_\mu$ also has an alternative representation
\begin{equation}\label{voicul}
\phi_{\mu}(z)=\gamma+\int_{\mathbb{R}}\frac{1+xz}{z-x}d\sigma(x).
\end{equation}
For future reference, we record the relation between the generating triple $(a,\eta,\rho)$ and the generating pair $(\gamma,\sigma)$ for the same measure $\mu$:
\begin{equation}\label{rel23}
\left\{
\begin{aligned}
&\sigma(dx)=a\delta_{0}(dx)+\frac{x^{2}}{1+x^{2}}\rho(dx)& \\
&\gamma=\eta-\int_{\mathbb{R}}x\left[\mathbf{1}_{[-1,1]}(x)-\frac{1}{1+x^{2}}\right]d\rho(x)&\\
\end{aligned}
\right.
\end{equation}
and, conversely,
\begin{equation}\label{rel32}
\left\{
\begin{aligned}
&a=\sigma(\{0\})& \\
&\eta=\gamma+\int_{\mathbb{R}\setminus\{0\}}\frac{1+x^{2}}{x}\left[\mathbf{1}_{[-1,1]}(x)-\frac{1}{1+x^{2}}\right]d\sigma(x)&\\
&\rho(dx)=\frac{1+x^{2}}{x^{2}}\mathbf{1}_{\mathbb{R}\setminus\{0\}}(x)\sigma(dx).&\\
\end{aligned}
\right.
\end{equation}
The following fundamental limit theorem was proved by Bercovici and Pata in \cite{BerPatDomains}.
\begin{theorem}
\label{main2}
For a sequence of probability measures $\set{\mu_{n}}$ and a strictly increasing sequence of positive integers $(k_{n})$, the following assertions are equivalent:
\begin{enumerate}
\item the sequence of $k_{n}$-fold free convolutions $\mu_{n}^{\boxplus k_{n}} $ converges weakly to a probability measure $\mu$;
\item there exist a finite positive Borel measure $\sigma$ on $\mathbb{R}$ and a real number $\gamma$ such that
\begin{equation}\label{main2.1}
k_{n}\frac{x^{2}}{x^{2}+1}d\mu_{n}(x)\overset{w.}{\rightarrow}d\sigma(x)
\end{equation}
and
\begin{equation}\label{main2.2}
\lim_{n\rightarrow\infty}k_{n}\int_{\mathbb{R}}\frac{x}{1+x^{2}}d\mu_{n}(x)=\gamma.
\end{equation}
\end{enumerate}
The pair of parameters $(\gamma,\sigma)$ comes from the Voiculescu transform (\ref{voicul}) of $\mu$. This also implies the $\boxplus$-infinite divisibility of $\mu$.
\end{theorem}

\subsection{Free Poisson Random Measures}

\begin{mydef}[Free Poisson Random Measures]\label{def_poissonrandom}
Let ($\Theta,\mathcal{E},\nu)$ be a measure space and put $\mathcal{E}_{0}=\{E\in \mathcal{E}:\nu(E)<\infty\}$. Let further $(\mathcal{A},\tau)$ be a $W^{\ast}-$probability space and let $\mathcal{A}_{+}$ denote the cone of positive operators in $\mathcal{A}$. A free Poisson random measure on $(\Theta,\mathcal{E},\nu)$ with values in $(\mathcal{A},\tau)$ is a mapping $M:\mathcal{E}_{0}\rightarrow\mathcal{A}_{+}$ with the following properties:
\begin{enumerate}
\item the distribution of $M(E)$ is a free Poisson distribution $\mathrm{Poiss}^{\boxplus}(\nu(E));$
\item for mutually disjoint sets $A_{1},...,A_{n}$ in $\mathcal{E}_{0}$, the random variables
\[
M(A_{1}),M(A_{2}),...,M(A_{n})
\]
are freely independent and $M(\cup_{j=1}^{n} A_{j})=\sum_{j=1}^{n} M(A_{j}).$
\end{enumerate}
\end{mydef}
Here, the free Poisson distribution $\mathrm{Poiss}^{\boxplus}(\lambda)$ is obtained by the limit in distribution of $$\left((1-\frac{\lambda}{N})\delta_{0}+\frac{\lambda}{N}\delta_{1}\right)^{\boxplus N},$$ as $N\rightarrow\infty$ (see Lecture 12 in \cite{Nica-Speicher-book}). The existence of free Poisson random measures is proved by Barndorff-Nielsen and Thorbj{\o}rnsen in \cite{BN-T-Levy-Ito}. For an alternative approach, see Remark~\ref{Remark:Compound-Poisson} below.

We next discuss integration with respect to a free Poisson random measure.

\begin{mydef}\label{integrationof_simple}
Let $s$ be a real-valued simple function in $L^{1}(\Theta,\mathcal{E},\nu)$ of the form $s=\sum_{j=1}^{r}a_{j}\mathbf{1}_{E_{j}},$ where $a_{j}\in\mathbb{R}\setminus\{0\}$ and $E_{j}$ are disjoint sets from $\mathcal{E}_{0}.$ Then, we define the integral of $s$ with respect to $M$ as
\[\int_{\Theta}sdM=\sum_{j=1}^{r}a_{j}M(E_{j})\in\mathcal{A}.\]
\end{mydef}
Because $M(E_{j})$ are positive in $\mathcal{A}$, the element $\int_{\Theta}sdM$ is self-adjoint in $\mathcal{A}$, for any real-valued simple function in $L^{1}(\Theta,\mathcal{E},\nu)$. Next, we can extend this integration to general functions in $L^{1}(\Theta,\mathcal{E},\nu)$.
\begin{lemma}
\cite[Proposition~4.3]{BN-T-Levy-Ito} Let $f$ be a real-valued function in the space $L^{1}(\Theta,\mathcal{E},\nu)$. Choose a sequence of real-valued simple functions $(s_{n})$ in $L^{1}(\Theta,\mathcal{E},\nu)$ which satisfies the assumptions of the Dominated Convergence Theorem, such that $s_{n}(\theta)\rightarrow f(\theta),$ for all $\theta\in\Theta.$ Then, $\int_{\Theta}s_n dM$ converges in probability to a self-adjoint (possibly unbounded) operator affiliated with $\mathcal{A}.$ This operator is independent of the choice of approximating sequence $(s_{n}).$ We denote this operator by $\int_{\Theta}fdM$.
\end{lemma}

The proof of the following lemma follows by the same techniques as Proposition~4.3 and Corollary~4.5 in \cite{BN-T-Levy-Ito}.

\begin{Lemma}
\label{Lemma:L1-convergence}
Let $f$ be a real-valued function in $L^{1}(\Theta,\mathcal{E},\nu)$. Choose a sequence of real-valued functions $(f_{n})$ in $L^{1}(\Theta,\mathcal{E},\nu)$ which satisfies the assumptions of the Dominated Convergence Theorem, such that $f_{n}(\theta)\rightarrow f(\theta),$ for all $\theta\in\Theta.$ Then, $\int_{\Theta}f_n dM$ converges in probability to $\int_{\Theta}fdM$.
\end{Lemma}

In fact, we only use a special measure space with a concrete intensity measure in our situation. Let $D=\mathbb{R_{+}}\times\mathbb{R}$ and $\mathcal{B}(D)$ be the set of all Borel subsets of $D.$ In our case,
\[
(\Theta,\mathcal{E},\nu)=(D,\mathcal{B}(D),Leb\otimes\rho),
\]
where $\rho$ is a L{\'e}vy measure. The free Poisson random measure $M$ that we will use is defined on $(D,\mathcal{B}(D),Leb\otimes\rho)$ with values in a $W^{\ast}-$probability space $(\mathcal{A},\tau)$. Besides, the integration with respect to this free Poisson measure $M$ we will use is also a special case.
\begin{lemma}\label{lemma422}
Let $\rho$ be a L{\'e}vy measure on the real line, and let $M$ be a free Poisson random measure on $(D,\mathcal{B}(D),Leb\otimes \rho)$ with values in the $W^{\ast}-$probability space ($\mathcal{A},\tau$). Suppose that $p(x)$ is any continuous function on $\mathbb{R}$.
\begin{enumerate}
\item For any $\epsilon>0$ and $0\leqslant s<t<\infty$, the integral $$\int_{(s,t]\times\{\epsilon<|x|\leqslant n\}}p(x)M(dt,dx)$$ converges in probability, as $n\rightarrow\infty,$ to some self-adjoint operator affiliated with $\mathcal{A}$, which is denoted by
$$\int_{(s,t]\times\{\epsilon<|x|<\infty\}}p(x)M(dt,dx).$$
\item If $\int_{[-1,1]}|p(x)|\rho(dx)<\infty,$ then for any $\epsilon>0$ and $0\leqslant s<t<\infty$, the integral $$\int_{(s,t]\times\{|x|\leqslant n\}}p(x)M(dt,dx)$$ converges in probability to some self-adjoint operator affiliated with $\mathcal{A},$ as $n\rightarrow\infty$. We denote it by $$\int_{(s,t]\times\mathbb{R}}p(x)M(dt,dx).$$\end{enumerate}
\end{lemma}

The statement of Lemma \ref{lemma422} is quite similar with Lemma 6.3 of \cite{BN-T-Levy-Ito}. In the paper \cite{BN-T-Levy-Ito}, the authors only proved the situation when $p(x)=x$ but their methods in Lemma 6.1 and Lemma 6.2 of \cite{BN-T-Levy-Ito} still work well for Lemma \ref{lemma422}. According to Lemma 6.3 of \cite{BN-T-Levy-Ito}, there are only two things for us to check. Since $\rho$ is a L{\'e}vy measure, we have that \[\int_{(s,t]\times \{\epsilon<|x|\leqslant n\}}|p(x)|Leb\otimes\rho(du,dx)=(t-s)\int_{\{\epsilon<|x|\leqslant n\}}|p(x)|\rho(dx)<\infty.\]If $\int_{[-1,1]}|p(x)|\rho(dx)<\infty,$ we have that
\begin{multline*}
\int_{(s,t]\times \{|x|\leqslant n\}}|p(x)|Leb\otimes\rho(du,dx) \\
=(t-s)\left[\int_{\{|x|\leqslant 1\}}|p(x)|\rho(dx)+\int_{\{1<|x|\leqslant n\}}|p(x)|\rho(dx)\right]<\infty.
\end{multline*}
Thus, integrals $\int_{(s,t]\times\{\epsilon<|x|\leqslant n\}}p(x)M(dt,dx)$ and $\int_{(s,t]\times\{|x|\leqslant n\}}p(x)M(dt,dx)$ are well-defined by Proposition 4.3 of \cite{BN-T-Levy-Ito}. Then, we can copy the proof of Lemma 6.3 of \cite{BN-T-Levy-Ito} and replace the function $f(x)=x$ by arbitrary continuous function $p(x)$ directly to prove Lemma  \ref{lemma422}. The idea for proving Lemma 6.3 is employing the Bercovici-Pata bijection to transform the statement into classical sense and then using Lebesgue's dominated convergence theorem.

\section{The Higher Variations of Free L{\'e}vy Processes}
\label{Sec:Single-variable}
\begin{Prop}\label{prop511}
If there exist a finite Borel measure $\sigma$ and a constant $\gamma$ such that\begin{equation}\label{condi1thm511}
N\frac{x^{2}}{x^{2}+1}d\mu_{N}(x)\overset{w.}{\rightarrow}d\sigma(x)
\end{equation}
and\begin{equation}\label{condi2thm511}
\lim_{N\rightarrow\infty}N\int_{\mathbb{R}}\frac{x}{1+x^{2}}d\mu_{N}(x)=\gamma,
\end{equation}
then there exists a family $\{\mu_{t}\}_{t\geqslant 0}$ of probability measures on $\mathbb{R}$ such that
\[
\mu_{N}^{\boxplus [Nt]}\overset{w.}{\rightarrow}\mu_{t},
\]
for any $t\in[0,\infty)$. Each $\mu_{t}$ is $\boxplus$-infinitely divisible and its Voiculescu transform is $\phi_{\mu_{t}}(z)=t\gamma+t\int_{\mathbb{R}}\frac{1+xz}{z-x}d\sigma(x)=t\phi_{\mu}(z),$ where $\mu:=\mu_{1}$ is the distribution of $X(1).$

Moreover, there exists a free L{\'e}vy process $\{X(t)\}_{t\geqslant 0}$ such that the distribution of each $X(t)$ is $\mu_{t}$, for all $t\geqslant 0$.
\end{Prop}
\begin{proof}
By Theorem \ref{main2}, we know that if there exist a finite Borel measure $\sigma$ and a constant $\gamma$ such that (\ref{condi1thm511}) and (\ref{condi2thm511}) hold, then $\mu_{N}^{\boxplus N}\overset{w.}{\rightarrow}\mu_{1}$. For any $ t\in[0,\infty),$ we have that
\[
[Nt]\frac{x^{2}}{x^{2}+1}d\mu_{N}(x)\overset{w.}{\rightarrow}td\sigma(x)=:d\sigma_{t}(x)
\]and\[
\lim_{N\rightarrow\infty}[Nt]\int_{\mathbb{R}}\frac{x}{1+x^{2}}d\mu_{N}(x)=t\lim_{N\rightarrow\infty}N\int_{\mathbb{R}}\frac{x}{1+x^{2}}d\mu_{N}(x)=t\gamma=:\gamma_{t}.
\]
Therefore, for any $t\in[0,\infty)$, there exists a probability measure $\mu_{t}$ such that $\mu_{N}^{\boxplus [Nt]}\overset{w.}{\rightarrow}\mu_{t}$. According to Theorem \ref{main2}, for any $t\in[0,\infty)$, $\mu_{t}$ is $\boxplus$-infinitely divisible since the Voiculescu transform of $\mu_{t}$ is $$\phi_{\mu_{t}}(z)=\gamma_{t}+\int_{\mathbb{R}}\frac{1+xz}{z-x}d\sigma_{t}(x)=t\phi_{\mu}(z),$$ where $\mu:=\mu_{1}$. Therefore, $\phi_{\mu_{t}}=\phi_{\mu_{t-s}}+\phi_{\mu_{s}},$ when $t>s\geqslant 0.$ In other words, $\mu_{t}=\mu_{t-s}\boxplus \mu_{s}.$ Meanwhile, $\phi_{\mu_{t}}\rightarrow 0$ when $t\rightarrow 0,$ which means $\mu_{t}\overset{w.}{\rightarrow}\delta_{0},$ as $t\rightarrow 0.$ Then, by Remark 6.7 in \cite{BN-T-Levy-Ito}, we can conclude that there exists a free L{\'e}vy process $\{X(t)\}_{t\geqslant 0}$, which is a family of self-adjoint operators affiliated with some $W^{\ast}$-probability space $(\mathcal{A}^{0},\tau^{0})$, such that the distribution of each $X(t)$ is $\mu_{t}$, for all $t\geqslant 0$.
\end{proof}

\begin{lemma}\label{changevaria}
Let $(\mathcal{A},\tau)$ be a $W^{\ast}$-probability space. Let $a \in \tilde{\mc{A}}_{sa}$ with distribution $\mu$, and $p(x)$ be a continuous real-valued function. Then the distribution $\mu^{(p)}$ of operator $p(a)$ (obtained via continuous functional calculus) can be obtained by the following formula:
\[\int_{\mathbb{R}}f(p(x))d\mu(x)=\int_{\mathbb{R}}f(x)d\mu^{(p)}(x),\]
for any bounded Borel function $f:\mathbb{R}\rightarrow\mathbb{R}$.
\end{lemma}
\begin{proof}
By definition of the distribution, for any bounded Borel function $f:\mathbb{R}\rightarrow\mathbb{R}$, \[\tau(f(a))=\int_{\mathbb{R}}f(x)d\mu(x).\]
Then, $f\circ p(x)$ is still a bounded Borel function. Thus,
\[\int_{\mathbb{R}}f(p(x))d\mu(x)=\tau\left(f(p(a))\right)=\int_{\mathbb{R}}f(x)d\mu^{(p)}(x). \qedhere\]
\end{proof}

Generally, Lemma \ref{changevaria} shows how to change variables between different probability measures.

Note the difference between the notation $\mu^{(p)}$ in the preceding lemma and $\rho^p$ in the following one.

\begin{lemma}\label{changvari-poisson}
Let $p(x)$ be any real-valued continuous function such that $p(0)=0$ and $p'(0)$ exists. Suppose that $M$ is a free Poisson random measure determined by a L{\'e}vy measure $\rho$ on the Borel measure space $(D,\mathcal{B}(D),Leb\otimes\rho)$ with values in some $W^{\ast}$-probability space $\mathcal{A}$. If $\rho^{p}$ is another measure defined by \begin{equation}\label{deflevymeasurechange}
\int_{\mathbb{R}}f(x)d\rho^{p}(x)=\int_{\mathbb{R}}f(p(x))\mathbf{1}_{\mathbb{R}\setminus\{0\}}(p(x))d\rho(x),
\end{equation}for any bounded Borel function $f(x)$ on $\mathbb{R},$ then $\rho^{p}$ is a L{\'e}vy measure.  The free Poisson random measure $M^{(p)}$ defined by $\rho^{p}$ on $(D,\mathcal{B}(D),Leb\otimes\rho^{p})$ has the following relation with $M$:\begin{equation}\label{changevariablerand1}
\int_{(0,t]\times\{\epsilon <|x|\}}xdM^{(p)}(t,x)\overset{d}{=}\int_{(0,t]\times\{\epsilon <|p(x)|\}}p(x)dM(t,x),
\end{equation} for any $t,\epsilon >0$, and
\begin{equation}\label{changevariablerand}
\int_{(0,t]\times\mathbb{R}}xdM^{(p)}(t,x)\overset{d}{=}\int_{(0,t]\times\mathbb{R}}p(x)dM(t,x), \quad \forall t> 0,
\end{equation}
provided that $\int_{[-1,1]}|x|d\rho^{p}(x)<\infty$.
\end{lemma}

\begin{proof}
Since $p(0)=0$, there exists an $\eps > 0$ such that $|p(x)|\leqslant 1$ when $\abs{x} \leqslant \eps$. Since $p'(0)$ exists, the function \[h(x):=\begin{cases}
\frac{p(x)}{x}, x\neq 0\\
p'(0), x=0,\\
\end{cases}\] is continuous on $\mathbb{R}$. First, we show that $\rho^{p}$ is a L{\'e}vy measure. If $f(x)=\mathbf{1}_{\{0\}}(x),$ then $\rho^{p}(\{0\})=\int_{\mathbb{R}}\mathbf{1}_{\{0\}}(x)d\rho^{p}(x)$ is zero by the definition (\ref{deflevymeasurechange}). Next, if $f(x)=\min\{1,x^{2}\},$ then we can get the following conclusion:
\begin{align*}
& \int_{\mathbb{R}}\min\{1,x^{2}\}d\rho^{p}(x) \\
&\quad=\int_{\mathbb{R}}\mathbf{1}_{[-1,1]}(x)x^{2}
d\rho^{p}(x)+\int_{\mathbb{R}}\mathbf{1}_{\mathbb{R}\setminus[-1,1]}(x)d\rho^{p}(x)&\\
&\quad =\int_{\mathbb{R}}\mathbf{1}_{[-1,1]\setminus\{0\}}(p(x))(p(x))^{2}
d\rho(x)+\int_{\mathbb{R}}\mathbf{1}_{\mathbb{R}\setminus[-1,1]}(p(x))d\rho(x)&\\
&\quad \leqslant \int_{\{x\in\mathbb{R}:\eps<|p(x)|\leqslant 1\}}p(x)^{2}
d\rho(x)+\int_{-\eps}^{\eps}h(x)^{2} x^2
d\rho(x)+\int_{\mathbb{R}\setminus[-\eps, \eps]}1d\rho(x)&\\
&\quad \leqslant \int_{\mathbb{R}\setminus[-\eps, \eps]}
1d\rho(x)+ \max_{-\eps \leqslant x \leqslant \eps} \abs{h(x)}^2 \int_{-1}^1 x^{2}
d\rho(x)+\int_{\mathbb{R}\setminus[-\eps, \eps]}1d\rho(x)<\infty.&
\end{align*}
Therefore, $\rho^{p}$ is a L{\'e}vy measure.

Second, we show that the relation (\ref{changevariablerand}) holds. If $\int_{[-1,1]}|x|d\rho^{p}(x)$ is finite, then so is $\int_{-1}^{1}|p(x)|d\rho(x)$, since
\begin{equation*}
\label{Eq:Integral-bounded}
\int_{-1}^{1}|x|d\rho^{p}(x)=\int_{\mathbb{R}}\mathbf{1}_{[-1,1]\setminus\{0\}}(p(x))\cdot |p(x)|d\rho(x)=\int_{-1}^{1}|p(x)|d\rho(x)<\infty.
\end{equation*}
Thus, the right-hand side and left-hand side of (\ref{changevariablerand}) make sense by Lemma \ref{lemma422}. According to Lemma \ref{lemma422}, we only need to show that \[\int_{(0,t]\times\{x:-n\leqslant x<n\}}xdM^{(p)}(t,x)\overset{d}{=}\int_{(0,t]\times\{x: -n\leqslant p(x)<n\}}p(x)dM(t,x),\]for all $t\geqslant 0$ and $n\in\mathbb{N}.$ For any $N\in\mathbb{N},$ consider mutually disjoint intervals
\[
E_{m}^{N}=\Bigl[-n+\frac{2n(m-1)}{N},-n+\frac{2nm}{N} \Bigr),
\]
where $1\leqslant m\leqslant N$ and $m\in\mathbb{N}.$ Then, the simple functions
\[
s_{N}(x)=\sum_{m=1}^{N}\left(-n+\frac{2n(m-1)}{N} \right)\mathbf{1}_{E_{m}^{N}}(x)
\]
converge to $f(x)=x$, for any $x\in[-n,n)$ as $N\rightarrow\infty.$ Thus,
\[
\int_{(0,t]\times\{x:-n\leqslant x<n\}}s_{N}(x)dM^{(p)}(t,x) \rightarrow \int_{(0,t]\times\{x:-n\leqslant x<n\}}xdM^{(p)}(t,x)
\]
in probability. Let $$J_{m}^{N}=\{x:p(x)\in E_{m}^{N}\},(1\leqslant m\leqslant N, m\in\mathbb{N}).$$ Then, $\cup_{m=1}^{N}J_{m}^{N}=\{x: -n\leqslant |p(x)|<n\}$ and $\{J_{m}^{N}\}$ are mutually disjoint. The simple functions
\[
g_{N}(x)=\sum_{m=1}^{N}\left(-n+\frac{2n(m-1)}{N}\right)\mathbf{1}_{J_{m}^{N}}(x) \rightarrow p(x)
\]
for any $x\in\{x: -n\leqslant |p(x)|<n\},$ as $N\rightarrow\infty.$ Therefore, when $N\rightarrow\infty$,
\[
\int_{(0,t]\times\{x: -n\leqslant |p(x)|<n\}}g_{N}(x)dM(t,x) \rightarrow \int_{(0,t]\times\{x: -n\leqslant |p(x)|<n\}}p(x)dM(t,x)
\]
in probability. We conclude that it suffices to show the equality in distribution
\[
\int_{(0,t]\times\{x: -n\leqslant |p(x)|<n\}}g_{N}(x)dM(t,x) \stackrel{d}{=} \int_{(0,t]\times\{-n\leqslant x<n\}}s_{N}(x)dM^{(p)}(t,x).
\]
Let $F_{m}^{N}=(0,t]\times E_{m}^{N}$ and $K_{m}^{N}=(0,t]\times J_{m}^{N}.$ By Definition \ref{integrationof_simple}, we know that \[\int_{(0,t]\times\{x:-n\leqslant x<n\}}s_{N}(x)dM^{(p)}(t,x)=\sum_{m=1}^{N}\left(-n+\frac{2n(m-1)}{N}\right)M^{(p)}(F_{m}^{N}),\]and\[\int_{(0,t]\times\{x: -n\leqslant p(x)<n\}}g_{N}(x)dM(t,x)=\sum_{m=1}^{N}\left(-n+\frac{2n(m-1)}{N}\right)M(K_{m}^{N}).\]
By Definition \ref{def_poissonrandom}, the distribution of $M(K_{m}^{N})$ is Poiss$^{\boxplus}(t\rho(J_{m}^{N}))$ and the distribution of $M^{(p)}(F_{m}^{N})$ is $\mathrm{Poiss}^{\boxplus}(t\rho^{p}(E_{m}^{N})).$ According to (\ref{deflevymeasurechange}), we conclude that, when $m\neq \frac{N}{2}+1$, $0\notin E_{m}^{N}, so$  \begin{align*}
\rho^{p}(E_{m}^{N})&=\int_{\mathbb{R}}\mathbf{1}_{E_{m}^{N}}(x)d\rho^{p}(x)&\\
&=\int_{\mathbb{R}}\mathbf{1}_{E_{m}^{N}\setminus\{0\}}(p(x))d\rho(x)&\\
&=\int_{\mathbb{R}}\mathbf{1}_{\{x:p(x)\in E_{m}^{N}\}}(x)d\rho(x)=\rho(J_{m}^{N}).&
\end{align*}
So, $\mathrm{Poiss}^{\boxplus}(t\rho(J_{m}^{N}))=\mathrm{Poiss}^{\boxplus}(t\rho^{p}(E_{m}^{N}))$, $m\neq \frac{N}{2}+1$. Notice that the coefficients in front of $M(K^{N}_{1+\frac{N}{2}})$ and $M^{(p)}(F^{N}_{1+\frac{N}{2}})$  are zero. Then, we get the final result (\ref{changevariablerand}). In general, for any $t,\epsilon>0$ and $n\in\mathbb{N},$ we can apply the same method and show that \[
\int_{(0,t]\times\{\epsilon <|x|<n\}}xdM^{(p)}(t,x)\overset{d}{=}\int_{(0,t]\times\{\epsilon <|p(x)|<n\}}p(x)dM(t,x),\] to prove equation (\ref{changevariablerand1}).
\end{proof}

\begin{lemma}
\label{Lemma:rho-p-integral}
Let $p(x)$ be any real-valued continuous function such that $p(0)=p'(0) = 0$ and $p''(0) = 2c$ exists. Then whether or not $\int_{-1}^{1}|x|d\rho(x)$ is finite, $\int_{-1}^{1}|x|d\rho^p(x)$ is finite.
\end{lemma}

\begin{proof}
Denote
\[q(x):=\begin{cases}
\frac{p(x)}{x^{2}},x\neq 0\\
c,x=0.
\end{cases}\]Then $q(x)$ is a continuous function. So we can check that\begin{align*}
& \int_{-1}^{1}|x|d\rho^{p}(x) \\
&\quad =\int_{\mathbb{R}}\mathbf{1}_{[-1,1]}(x)|x|d\rho^{p}(x) \\
		&\quad =\int_{\mathbb{R}}\mathbf{1}_{[-1,1]\setminus\{0\}}(p(x))|p(x)|d\rho(x)&\\
		&\quad =\int_{-1}^{1}\mathbf{1}_{[-1,1]\setminus\{0\}}(p(x))|p(x)|d\rho(x)+\int_{\mathbb{R}\setminus[-1,1]}\mathbf{1}_{[-1,1]\setminus\{0\}}(p(x))|p(x)|d\rho(x)&\\
		&\quad \leqslant\|q\|_{C([-1,1])}\int_{-1}^{1}x^{2}d\rho(x)+\int_{\mathbb{R}\setminus[-1,1]}\mathbf{1}_{\{0<|p(x)|\leqslant 1\}}(x)|p(x)|d\rho(x)\\
		&\quad \leqslant C\int_{\mathbb{R}}\min\{1,x^{2}\}d\rho(x)<\infty.& \qedhere
		\end{align*}
\end{proof}

Theorem~\ref{Thm:Variations-distribution} follows from the following more general result by taking $p(x) = x^k$.

\begin{theorem}
\label{Thm:Conergence-p}
For each $N \in \mf{N}$, let $\set{X_{i, N} : i \in \mf{N}}$ be free, identically distributed, self-adjoint random variables affiliated to $(\mc{A}, \tau)$. Suppose that for $t \geqslant 0$,
\[\sum_{r=1}^{[Nt]}X_{N,r}\overset{d.}{\rightarrow}X(t),\] where $X(t)$ is a free L{\'e}vy process.  Let $p(x)$ be any real-valued continuous function such that $p(0)=0$, $p'(0)=b$, and $p''(0)=2c$.
Then, there exists a L{\'e}vy process $X^{p}(t)$ such that
\begin{equation}\label{thm7.4.1}
\sum_{r=1}^{[Nt]}p\left(X_{N,r}\right)\overset{d.}{\rightarrow}X^{p}(t).
\end{equation}
In addition, if $X(t)$ has the L{\'e}vy-It{\^o} decomposition (\ref{levyde}) with the generating triple $(\eta, a, \rho)$, then $X^{p}(t)$ has a representation in the form:
\begin{equation}\label{thm7.4.3}
X^{p}(t)\overset{d}{=} b X(t) + a c t \mathbf{1}_{\mathcal{A}^{0}} +  \int_{(0,t]\times\mathbb{R}}(p(x) - b x) dM(t,x),
\end{equation}
where $M$ is a free Poisson random measure coming from the L{\'e}vy-It{\^o} decomposition of $X(t)$. This is the case whether or not $\int_{-1}^{1}|x|d\rho(x)$ is finite. In particular, if $p'(0) = 0$,
\begin{equation*}\label{thm7.4.3}
X^{p}(t)\overset{d}{=} a c t \mathbf{1}_{\mathcal{A}^{0}} + \int_{(0,t]\times\mathbb{R}}p(x) dM(t,x).
\end{equation*}
\end{theorem}
\begin{proof}
Let $X(1)$ be generated by the pair $(\gamma,\sigma)$.
Let $\mu_{N}$ and $\mu_{N}^{p}$ be the distributions of $X_{N,r}$ and $p(X_{N,r})$ respectively. Recall that by Lemma \ref{changevaria},
\[
\int_{\mathbb{R}}f(x)d\mu_{N}^{p}(x)=\int_{\mathbb{R}}f(p(x))d\mu_{N}(x),
\]
for any real-valued and bounded Borel function $f(x)$. Let
\[q(x):=\begin{cases}
\frac{p(x)-bx}{x^{2}},x\neq 0\\
c,x=0.
\end{cases}\]Then $q(x)$ is a continuous function. Therefore, \begin{align*}
&\lim_{N\rightarrow\infty}[Nt]\int_{\mathbb{R}}\frac{x}{1+x^{2}}d\mu_{N}^{p}(x)=t\lim_{N\rightarrow\infty}N\int_{\mathbb{R}}\frac{p(x)}{1+p(x)^{2}}d\mu_{N}(x)&\\
&=t\lim_{N\rightarrow\infty}N\Bigl[\int_{\mathbb{R}}\frac{bx}{1+x^{2}}d\mu_{N}(x)+\int_{\mathbb{R}}\left(\frac{p(x)}{1+p(x)^{2}}-\frac{bx}{1+x^{2}}\right)d\mu_{N}(x)\Bigr]&\\
&=tb\gamma+t\lim_{N\rightarrow\infty}N\int_{\mathbb{R}}g_{p}(x)\frac{x^{2}}{1+x^{2}}d\mu_{N}(x),&\\
\end{align*}
where $g_{p}(x)=\frac{p (x)+ q(x) - b (b + xq(x)) p(x)}{1 + p(x)^2}\in Cb(\mathbb{R})$ and $g_{p}(0)=c$. So $\gamma^{p}$ is defined by \[\gamma^{p}:=\lim_{N\rightarrow\infty}N\int_{\mathbb{R}}\frac{x}{1+x^{2}}d\mu_{N}^{p}(x)=b\gamma+\int_{\mathbb{R}}g_{p}(x)d\sigma(x),\]
where $\gamma$ and $\sigma$ are defined by \eqref{main2.1} and \eqref{main2.2}. Define \[h(x):=\begin{cases}
\frac{p(x)}{x},x\neq 0\\
b,x=0.
\end{cases}\] Then, $h(x)$ is a continuous function and $xh(x)=p(x)$. For any $f(x)\in Cb(\mathbb{R}),$
\begin{align*}
& [Nt]\int_{\mathbb{R}}f(x)\frac{x^{2}}{x^{2}+1}d\mu_{N}^{p}(x) \\
&\quad =[Nt]\int_{\mathbb{R}}f(p(x))\frac{p(x)^{2}}{p(x)^{2}+1}d\mu_{N}(x)&\\
&\quad =[Nt]\int_{\mathbb{R}}f(p(x))\frac{p(x)^{2}}{p(x)^{2}+1}\frac{x^{2}}{x^{2}+1}\frac{x^{2}+1}{x^{2}}d\mu_{N}(x)&\\
&\quad \overset{N\rightarrow\infty}{\longrightarrow}t\int_{\mathbb{R}}f(p(x))\frac{p(x)^{2}+h(x)^{2}}{p(x)^{2}+1}d\sigma(x).
\end{align*}
Let $h_{p}(x):=\frac{p(x)^{2}+h(x)^{2}}{p(x)^{2}+1}$, which is a positive bounded Borel function on $\mathbb{R}$. We denote by $d\widetilde{\sigma}(x)$ the measure $h_{p}(x)d\sigma(x)$. The measure $d\sigma^{p}(x)$ is defined by $\int_{\mathbb{R}}f(x)d\sigma^{p}(x)=\int_{\mathbb{R}}f(p(x))d\widetilde{\sigma}(x)$, for any bounded Borel function $f(x)$. Then, \[N\frac{x^{2}}{x^{2}+1}d\mu_{N}^{p}(x)\overset{w.}{\rightarrow}d\sigma^{p}(x),\] as $N\rightarrow\infty$. Since $\sigma$ is a finite positive Borel measure on $\mathbb{R}$, we know that $\sigma^{p}$ is also a finite positive Borel measure. Thus, the conclusion (\ref{thm7.4.1}) follows immediately from Theorem \ref{main2}. By Proposition \ref{prop511}, we know that $\{X^{p}(t)\}_{t\geqslant 0}$ can be a free L{\'e}vy process affiliated with some $W^{\ast}$-probability space. Denote the free generating triplet of $X^{p}(1)$ by $(a^{p},\eta^{p},\rho^{p})$.

Next, based on Theorem \ref{Thm:Levy-Ito}, L{\'e}vy process $X^{p}(t)$ has a decomposition in the form of (\ref{levyde}) with free generating triplet $(a^{p},\eta^{p},\rho^{p})$. Hence, to prove the representation (\ref{thm7.4.3}) of $X^{p}(t)$, it is necessary to compute the free generating triplet $(a^{p},\eta^{p},\rho^{p})$ in terms of free generating pair $(\gamma,\sigma)$ or free generating triplet $(a,\eta,\rho)$ of $X(t)$. Firstly, $a^{p}=\sigma^{p}(\{0\})=\int_{\mathbb{R}}\mathbf{1}_{\{0\}}(x)d\sigma^{p}(x)=\sigma(\{0\})h(0)^{2}=ab^{2}.$ Secondly, for L{\'e}vy measure $\rho^{p}$ and any bounded Borel function $f(x)$, we have that
\begin{align*}
\int_{\mathbb{R}}f(x)d\rho^{p}(x)&=\int_{\mathbb{R}}f(x)\frac{1+x^{2}}{x^{2}}\mathbf{1}_{\mathbb{R}\setminus\{0\}}(x)d\sigma^{p}(x)&\\
&=\int_{\mathbb{R}}f(p(x))\frac{1+(p(x))^{2}}{p(x)^{2}}\mathbf{1}_{\mathbb{R}\setminus\{0\}}(p(x))\frac{p(x)^{2}+h(x)^{2}}{1+p(x)^{2}}d\sigma(x)&\\
&=\int_{\mathbb{R}}f(p(x))\mathbf{1}_{\mathbb{R}\setminus\{0\}}(p(x))\frac{1+x^{2}}{x^{2}}d\sigma(x)&\\
&=\int_{\mathbb{R}}f(p(x))\mathbf{1}_{\mathbb{R}\setminus\{0\}}(p(x))d\rho(x).
\end{align*}
Therefore, $\rho^{p}$ is precisely the measure from Lemma \ref{changvari-poisson}, and in particular a L{\'e}vy measure. Thirdly, by the relation $\eta^{p}=\gamma^{p}+\int_{\mathbb{R}\setminus\{0\}}\frac{1+x^{2}}{x}(\mathbf{1}_{[-1,1]}(x)-\frac{1}{1+x^{2}})d\sigma^{p}(x),$ and the corresponding relation between $\eta$ and $\gamma$, using also $a=\sigma(\{0\})$ and relation (\ref{rel23}), we can deduce that
\begin{align*}
\eta^{p}&=b\gamma+\int_{\mathbb{R}}\mathbf{1}_{\{p(x)\neq0\}}(x)g_{p}(x)d\sigma(x)+\int_{\mathbb{R}}\mathbf{1}_{\{p(x)=0\}}(x)g_{p}(x)d\sigma(x)&\\
&\quad +\int_{\mathbb{R}}\mathbf{1}_{\{p(x)\neq0\}}(x)\frac{h^{2}+p^{2}}{p}\left(\mathbf{1}_{\{-1\leqslant p(x)\leqslant 1\}}(x)-\frac{1}{1+(p(x))^{2}}\right)d\sigma(x)&\\
&=b\gamma - \int_{\mathbb{R}\setminus\{0\}} \mathbf{1}_{\{p(x)=0\}}(x)\frac{b}{x}d\sigma(x) + \int_{\mathbb{R}} \mathbf{1}_{\{x= 0\}}(x)c d\sigma(x) \\
& \quad +\lim_{\epsilon\searrow 0}\left[\int_{\{\epsilon<|p(x)|\}}\left(\frac{h^{2}+p^{2}}{p}\mathbf{1}_{\{-1\leqslant p\leqslant 1\}}(x)  \right)d\sigma-\int_{\{|p(x)|>\epsilon\}}\frac{b}{x}d\sigma(x)\right]&\\
& =b\gamma+ac+ \lim_{\epsilon\searrow 0}\left[\int_{\mathbb{R}}\mathbf{1}_{\{\epsilon<|p(x)|\leqslant 1\}}(x)\frac{h^{2}+p^{2}}{p}d\sigma(x)-\int_{\{|x|>\epsilon\}}\frac{b}{x}d\sigma(x)\right]&\\
&=b\eta+ac+ \left( \int_{\mf{R}} \mathbf{1}_{\set{0 < \abs{p(x)} \leqslant 1}}(x) p(x) - \mathbf{1}_{\set{0 < x \leqslant 1}}(x) b x \right) \,d\rho(x).
\end{align*}
Note that for some $\eps > 0$, $\abs{p(x)} \leqslant 1$ for $\abs{x} \leqslant \eps$. So
\[
\begin{split}
& \int_{\mf{R}} \abs{ \mathbf{1}_{\set{0 < \abs{p(x)} \leqslant 1}}(x) p(x) - \mathbf{1}_{\set{0 < x \leqslant 1}}(x) b x} \,d\rho(x) \\
&\quad = \int_{-\eps}^\eps \abs{q(x)} x^2 \,d\rho(x) \\
&\quad\qquad + \int_{\mf{R}} \abs{ \mathbf{1}_{\set{0 < \abs{p(x) \leqslant 1}, \abs{x} > \eps}}(x) p(x) - \mathbf{1}_{\set{\eps < x \leqslant 1}}(x) b x} \,d\rho(x) \\
&\quad \leqslant \sup_{-\eps \leqslant x \leqslant \eps} \abs{q(x)} \int_{-\eps}^\eps x^2 \,d\rho(x) + 2 \int_{\set{\abs{x} > \eps}} \,d\rho(x) < \infty
\end{split}
\]
since $\rho$ is a L{\'e}vy measure, and so the expression above makes sense.

Combine three results we got above and recall the general free L{\'e}vy-It{\^o} decomposition of $X^{p}(t)$ with the free generating triplet $(a^{p},\eta^{p},\rho^{p}).$ Let $M^{(p)}$ be the free Poisson random measure on $(D,\mathcal{B}(D),Leb\otimes\rho^{p}).$ Then, we can simplify the last part of L{\'e}vy-It{\^o} decomposition of $X^{p}(t)$ with respect to the free Poisson random measure $M^{(p)}$:
\begin{align*}
&\lim_{\epsilon\searrow 0}\left[\int_{(0,t]\times\{|x|>\epsilon\}}xdM^{(p)}(t,x)-\int_{(0,t]\times\{\epsilon<|x|\leqslant 1\}}xLeb\otimes\rho^{p}(dt,dx)\mathbf{1}_{\mathcal{A}^{0}}\right]&\\
=&\lim_{\epsilon\searrow 0}\left[\int_{(0,t]\times\{|p(x)|>\epsilon\}}p(x)dM(t,x)-t\int_{\mathbb{R}}x\mathbf{1}_{\{\epsilon<|x|\leqslant 1\}}(x)d\rho^{p}(x)\mathbf{1}_{\mathcal{A}^{0}}\right]&\\
=&\lim_{\epsilon\searrow 0}\left[\int_{(0,t]\times\{|p(x)|>\epsilon\}}p(x)dM(t,x)-t\int_{\mathbb{R}}p(x)\mathbf{1}_{\{\epsilon<|x|\leqslant 1\}}(p(x))d\rho(x)\mathbf{1}_{\mathcal{A}^{0}}\right]&.
\end{align*}
Here, we employ Lemma \ref{changvari-poisson}, integration by substitution with respect to free Poisson random measures and relation (\ref{rel32}). Thus finally,
\[
\begin{split}
& X^{p}(t) \\
&\quad\overset{d}{=}
t \left( b\eta+ac+ \lim_{\epsilon\searrow 0} \left( \int_{\mf{R}} \mathbf{1}_{\set{\eps < \abs{p(x)} \leqslant 1}}(x) p(x) - \mathbf{1}_{\set{\eps < x \leqslant 1}}(x) b x \right) \,d\rho(x) \right) \mathbf{1}_{\mathcal{A}^{0}} \\
&\quad\quad + \sqrt{a} b S(t) \\
&\quad\quad + \lim_{\epsilon\searrow 0}\left[\int_{(0,t]\times\{|p(x)|>\epsilon\}}p(x)dM(t,x)-t\int_{\mathbb{R}}\mathbf{1}_{\{\epsilon<|p(x)|\leqslant 1\}}(x) p(x)d\rho(x)\mathbf{1}_{\mathcal{A}^{0}}\right] \\
&\quad = b \Bigl[\eta t \mathbf{1}_{\mathcal{A}^{0}}+ \sqrt{a} S(t)
 \\
&\quad\quad +\lim_{\epsilon\searrow 0} \Bigl(\int_{(0,t]\times \set{|x|>\epsilon}}xdM(t,x)
- t \int_{\set{\epsilon<|x|\leqslant 1}}x \,d\rho(x)\mathbf{1}_{\mathcal{A}^{0}} \Bigr) \Bigr] \\
&\quad\quad + a c t + \lim_{\epsilon\searrow 0} \Bigl( \int_{(0,t]\times\{|p(x)|>\epsilon\}}p(x)dM(t,x) - b \int_{(0,t]\times \set{|x|>\epsilon}}xdM(t,x)  \Bigr) \\
&\quad = b X(t) + a c t + \lim_{\epsilon\searrow 0} \int_{(0,t]\times \set{|x|>\epsilon}} (p(x) - b x) dM(t,x) \\
&\quad\quad + \lim_{\epsilon\searrow 0} \int_{(0,t]\times \mf{R}} p(x) (\mathbf{1}_{\{|p(x)|>\epsilon\}} - \mathbf{1}_{\set{|x|>\epsilon}}) dM(t,x).
\end{split}
\]
Here we used the fact that the distribution of $S(t)$ is symmetric. Since
\[
\int_{(0,t]\times \mf{R}} (p(x) - b x) dM(t,x) = \int_{(0,t]\times \mf{R}} x dM^{(p(x) - b x)}(t,x)
\]
exists by Lemmas~\ref{Lemma:rho-p-integral} and \ref{lemma422}, and the functions $(p(x) - b x) \mathbf{1}_{\abs{x} \leqslant \eps}$ have a uniform integrable bound and converge to zero pointwise as $\eps \rightarrow 0$, by Lemma~\ref{Lemma:L1-convergence} we have
\[
\begin{split}
& \lim_{\epsilon\searrow 0} \int_{(0,t]\times \set{|x|>\epsilon}} (p(x) - b x) dM(t,x) \\
&\quad = \int_{(0,t]\times \mf{R}} (p(x) - b x) dM(t,x) - \lim_{\epsilon\searrow 0} \int_{(0,t]\times \set{|x|\leq\epsilon}} (p(x) - b x) dM(t,x) \\
&\quad = \int_{(0,t]\times \mf{R}} (p(x) - b x) dM(t,x).
\end{split}
\]
Finally, the functions
\[
p(x) (\mathbf{1}_{\{|p(x)|>\epsilon\}} - \mathbf{1}_{\set{|x|>\epsilon}}) = - p(x) (\mathbf{1}_{\{|p(x)|\leq\epsilon\}} - \mathbf{1}_{\set{|x|\leq\epsilon}})
\]
also have a uniform integrable bound and converge to zero pointwise as $\eps \rightarrow 0$. Therefore by Lemma~\ref{Lemma:L1-convergence},
\[
\lim_{\epsilon\searrow 0} \int_{(0,t]\times \mf{R}} p(x) (\mathbf{1}_{\{|p(x)|>\epsilon\}} - \mathbf{1}_{\set{|x|>\epsilon}}) dM(t,x) = 0. \qedhere
\]
\end{proof}

\begin{Remark}
It is natural to consider, more generally, free additive (not necessarily) stationary processes approximated by free, non-identically distributed triangular arrays which are infinitesimal, that is, their distributions $\mu_{i, N}$ satisfy
\[
\lim_{N \rightarrow \infty} \max_{1 \leqslant i \leqslant k_N} \mu_{i, N} (\set{\abs{x} \geqslant \eps}) = 0
\]
for every $\eps > 0$. The following very simple example shows how without additional assumptions, the results immediately break down. Let
\[
X_{i, N} = \frac{1}{N} + (-1)^i \frac{1}{N^\alpha}, i = 1, \ldots, 2N.
\]
Then clearly the array $\set{X_{i, N}}$ is infinitesimal, and $\lim_{N \rightarrow \infty} \sum_{i=1}^{[2N t]} X_{i, N} = 2 t$. But
\[
\sum_{i=1}^{[2N t]} X_{i, N}^2 \sim \frac{2t }{N^{2 \alpha - 1}}
\]
diverges for $\alpha < \frac{1}{2}$. So while the quadratic variation of a non-random process is zero, these sums do not converge to it. Compare with the remarks on page 494 of \cite{Avram-Taqqu-symmetric-poly}.
\end{Remark}

\section{Convergence in moments}
\label{Sec:Moments}

For a non-crossing partition $\pi \in \NC(n)$, denote
\[
\tau_\pi\left[ a_1, \ldots, a_n \right] = \prod_{V \in \pi} \tau\left[\prod_{i \in V} a_i \right].
\]
Recall that the free cumulant functional is defined by
\[
R[a_1, \ldots, a_n] = \sum_{\pi \in \NC(n)} \mathrm{M\ddot{o}b}(\pi) \tau_\pi[a_1, \ldots, a_n],
\]
where $\mathrm{M\ddot{o}b}$ is the M{\"o}bius function on the lattice of non-crossing partitions. The key property of the free cumulant functional is that if $a_1, \ldots, a_k$ are free, then
\[
R[a_{u(1)}, \ldots, a_{u(n)}] = 0
\]
unless $u(1) = \ldots = u(n)$.

\begin{proof}[Proof of Theorem~\ref{Thm:Joint-moments}]
Note first that by freeness and the free moment-cumulant formula,
\[
\begin{split}
& R\left( \sum_{i=1}^{[Nt]} X_{i, N}^{u(1)}, \ldots, \sum_{i=1}^{[Nt]} X_{i, N}^{u(k)} \right) - \tau \left[ \sum_{i=1}^{[Nt]} X_{i, N}^{u(1) + \ldots + u(k)} \right] \\
&\quad = \sum_{i=1}^{[Nt]} \left( R\left(X_{i, N}^{u(1)}, \ldots, X_{i, N}^{u(k)} \right) - \tau \left[X_{i, N}^{u(1) + \ldots + u(k)} \right] \right) \\
&\quad = \sum_{i=1}^{[Nt]} \sum_{\substack{\pi \in \NC(k) \\ \pi \neq \hat{1}_k}} \mathrm{M\ddot{o}b}(\pi) \tau_\pi \left[X_{i, N}^{u(1)}, \ldots, X_{i, N}^{u(k)} \right]. \end{split}
\]
The absolute value of this expression is bounded by
\[
\begin{split}
& \sum_{\substack{\pi \in \NC(k) \\ \pi \neq \hat{1}_k}} \abs{\mathrm{M\ddot{o}b}(\pi)} \abs{ \sum_{i=1}^{[Nt]} \prod_{V \in \pi} \tau \left[ X_{i, N}^{\sum_{j \in V} u(j)} \right] } \\
&\leqslant \sum_{\substack{\pi \in \NC(k) \\ \pi \neq \hat{1}_k}} \abs{\mathrm{M\ddot{o}b}(\pi)} \left( \sum_{i=1}^{[Nt]} \tau \left[ X_{i, N}^{\sum_{j \in V_1} u(j)} \right]^2  \right)^{1/2} \left(\sum_{i=1}^{[Nt]} \tau \left[ X_{i, N}^{\sum_{j \in V_2} u(j)} \right]^2 \right)^{1/2} \\
&\qquad \times \prod_{V \in \pi \setminus \set{V_1, V_2}} \max_{1 \leqslant i \leqslant [N t]} \abs{X_{i,N}^{\sum_{j \in V} u(j)}},
\end{split}
\]
which goes to zero as $N \rightarrow \infty$, by assumption. So to prove joint convergence in moments, it suffices to show that the limit
\[
\lim_{N \rightarrow \infty} \tau \left[ \sum_{i=1}^{[Nt]} X_{i, N}^{k} \right]
\]
exists for each $k$. Indeed, applying the derivation above to the case $u(1) = \ldots = u(k) = 1$,
\[
\tau \left[ \sum_{i=1}^{[Nt]} X_{i, N}^{k} \right] - R_k\left( \sum_{i=1}^{[Nt]} X_{i, N} \right) \rightarrow 0
\]
as $N \rightarrow \infty$. Finally, by assumption
\[
R_k\left( \sum_{i=1}^{[Nt]} X_{i, N} \right)
\rightarrow R_k (X(t)).
\]
The statement about processes follows as in Proposition~\ref{prop511}.
\end{proof}

\section{Convergence in probability}
\label{Sec:Probability}

We first quote a result from \cite{BV93}.

\begin{Lemma}[Lemma~4.4]
\label{Lemma:One-sided-compression}
Let $(\mc{A}, \tau)$ be a $W^\ast$-probability space, $T_1, T_2, T_1', T_2' \in \tilde{\mc{A}}$, and $p_1, p_2 \in \mc{A}$ orthogonal projections. Suppose $T_j' = T_j p_j$, for $j = 1, 2$. Then there exist projections $p, q \in \mc{A}$ such that
\begin{enumerate}
\item
$(T_1 T_2) p = (T_1' T_2') p$,
\item
$(T_1 + T_2) q = (T_1' + T_2')q$, and
\item
$\tau[p], \tau[q] \geqslant \tau[p_1] + \tau[p_2] - 1$.
\end{enumerate}
\end{Lemma}

\begin{Remark}
\label{Remark:Compound-Poisson}
Let $\rho$ be a probability measure on $\mf{R}$. In the tracial non-commutative probability space $\mc{C} = L^\infty((0,1] \times \mf{R}, \mathrm{Leb} \otimes \rho)$, consider the projections $P(B) = \chi_B$ for every Borel set $B$. Let $s$ be a semicircular element free from $\mc{C}$. Then according to \cite{Nica-Speicher-Multiplication}, the family of operators $M : B \rightarrow s P(B) s$ satisfies all the properties of a free Poisson random measure in Definition~\ref{def_poissonrandom}. Next, let
\[
e(t) = \int_{\mf{R}} x \,P((0,t] \times dx),
\]
meaning that the spectral projections of $e_t$ are $\set{P((0,t] \times (-\infty, x))}$. Then
\[
\set{e(t) : t \in (0,1]}
\]
is a process with orthogonal increments, and $\set{s e(t) s : t \in (0,1]}$ is a free compound Poisson process. Note that
\[
s e(t) s = \int_{\mf{R}} x s \,P((0,t] \times dx) s = \int_{(0,t] \times \mf{R}} x \,dM(t,x).
\]
\end{Remark}

\begin{Prop}
\label{Prop:Compression-zero}
Let $Z_1, \ldots, Z_k$ be bounded and centered, free from a stationary process $\set{e(t)}$ with orthogonal increments. Then
\[
\sum_{i=1}^N e_{i, N}^{m_0} Z_1 e_{i, N}^{m_1} Z_2 \ldots e_{i, N}^{m_{k-1}} Z_k e_{i,N}^{m_k} \rightarrow 0
\]
in probability as $N \rightarrow \infty$. Here we denote as usual $e_{i,N} = e(\frac{i}{N}) - e(\frac{i-1}{N})$.
\end{Prop}

\begin{proof}
Without loss of generality, assume that $\set{e(t)}$ has the form in Remark~\ref{Remark:Compound-Poisson}. For arbitrary $\eps > 0$, choose $T$ so that $\rho((-T, T)^c) < \eps$. Denote
\[
q_{i, N} = \chi_{\left((0,1] \times \mf{R} \right) \setminus \left((\frac{i-1}{N}, \frac{i}{N}] \times (-T, T)^c \right)},
\]
so that $\tau(q_{i, N}) > 1 - \eps/N$. Denote $e_{i, N}' = e_{i, N} q_{i,N}$. Then $\set{e_{i, N}' : 1 \leq i \leq N}$ are still orthogonal, and
\[
\norm{\sum_{i=1}^N e_{i,N}'} = \norm{\int_{-T}^T x \,P((0,1] \times dx)} \leq T.
\]
According to Lemma~\ref{Lemma:One-sided-compression}, there is a projection $p_{i,N}$ with
\[
\tau(p_{i,N}) > 1 - \sum_{j=0}^k m_j \eps/N
\]
such that
\begin{multline*}
e_{i, N}^{m_0} Z_1 e_{i, N}^{m_1} Z_2 \ldots e_{i, N}^{m_{k-1}} Z_k e_{i,N}^{m_k} p_{i,N} \\
= (e'_{i, N})^{m_0} Z_1 (e'_{i, N})^{m_1} Z_2 \ldots (e'_{i, N})^{m_{k-1}} Z_k (e'_{i,N})^{m_k} p_{i,N}.
\end{multline*}
Therefore for $p_N = \bigwedge_{i=1}^N p_{i,N}$, $\tau(p_N) > 1 - \sum_{j=0}^k m_j \eps$ and
\begin{multline*}
\sum_{i=1}^N e_{i, N}^{m_0} Z_1 e_{i, N}^{m_1} Z_2 \ldots e_{i, N}^{m_{k-1}} Z_k e_{i,N}^{m_k} p_{N} \\
= \sum_{i=1}^N (e'_{i, N})^{m_0} Z_1 (e'_{i, N})^{m_1} Z_2 \ldots (e'_{i, N})^{m_{k-1}} Z_k (e'_{i,N})^{m_k} p_{N}.
\end{multline*}
On the other hand, according to Theorem~3 from \cite{AnsFSM},
\begin{multline*}
\norm{\sum_{i=1}^N (e'_{i, N})^{m_0} Z_1 (e'_{i, N})^{m_1} Z_2 \ldots (e'_{i, N})^{m_{k-1}} Z_k (e'_{i,N})^{m_k}} \\
\leqslant 4^{2k} (\max \norm{Z_j})^k T^{\sum_{j=0}^k m_j} N^{-k/2}.
\end{multline*}
The result follows.
\end{proof}

\begin{proof}[Proof of Theorem~\ref{Thm:Compound-Poisson}]
By using the addition part of Lemma~\ref{Lemma:product-probability}, we may assume that $t \in (0, 1]$. Note first that by Lemma~\ref{Lemma:L1-convergence},
\[
\int_{\bigl(0, \frac{[N t]}{N}\bigr] \times \mf{R}} x^k \,dM(t,x) \rightarrow \int_{(0, t] \times \mf{R}} x^k \,dM(t,x)
\]
in probability as $N \rightarrow \infty$. Next, write $X(t) = s e(t) s$ as before. By the same reasoning as in Remark~\ref{Remark:Compound-Poisson},
\[
\int_{(0, t] \times \mf{R}} x^k \,dM(t,x) = s e(t)^k s.
\]
Therefore
\[
\begin{split}
& \sum_{i=1}^{[N t]} X_{i, N}^k - \int_{\bigl(0, \frac{[N t]}{N}\bigr] \times \mf{R}} x^k \,dM(t,x)
= \sum_{i=1}^{[N t]} (s e_{i,N} s)^k  - \sum_{i=1}^{[N t]} s e_{i,N}^k s \\
&\quad = \sum_{j=1}^{k-1} \sum_{\substack{m_0, m_1, \ldots, m_j \geqslant 1 \\ m_0 + m_1 + \ldots + m_j = k}}  s \left( \sum_{i=1}^{[N t]} e_{i, N}^{m_0} (s^2 - 1) \ldots e_{i, N}^{m_{j-1}} (s^2 - 1) e_{i,N}^{m_j} \right) s.
\end{split}
\]
Now note that $\tau(s^2 - 1) = 0$ and apply Proposition~\ref{Prop:Compression-zero}.
\end{proof}

check

\begin{proof}[Proof of Corollary~\ref{Cor:Joint-probability}]
Combining Theorem~\ref{Thm:Compound-Poisson} with Lemma~\ref{Lemma:product-probability}, polynomials in the variables $\set{\sum_{i=1}^{[N t]} X_{i, N}^j}$ converge to the corresponding polynomials in $\set{X^{(j)}(t)}$ in probability. Finally, by Proposition~2.19 in \cite{BNTSelf} (see also Proposition~2.1 in \cite{Lindsay-Pata-WLLN}), convergence in probability implies convergence in distribution.
\end{proof}

\begin{proof}[Proof of Corollary~\ref{Cor:Convergence-integral}]
According to Corollary~\ref{Cor:Distinct},
\begin{multline*}
\sum_{\substack{1 \leqslant i(1), i(2), \ldots, i(k) \leqslant [N t] \\ i(1) \neq i(2), i(2) \neq i(3), \ldots, i(k-1) \neq i(k)}} X_{i(1), N} X_{i(2), N} \ldots X_{i(k), N} \\
= \sum_{j=1}^k (-1)^{k - j} \sum_{\substack{m_1, \ldots, m_j \geqslant 1 \\ m_1 + \ldots + m_j = k}} \left( \sum_{i(1)=1}^{[N t]} X_{i(1), N}^{m_1} \right) \ldots \left( \sum_{i(j)=1}^{[N t]} X_{i(j), N}^{m_j} \right).
\end{multline*}
Now apply either Theorem~\ref{Thm:Joint-moments} or Corollary~\ref{Cor:Joint-probability}.
\end{proof}

See the second author's thesis for a direct proof.

\begin{Remark}
In the case of a process which is not necessarily centered, normalizing it so that $\tau[X(t)] = t$, a more natural definition of an $n$-fold stochastic integral $\psi_n$, according to Theorem~4 of \cite{AnsFSM}, is
\[
\psi_n = X \psi_{n-1} + \sum_{j=2}^n (-1)^{j-1} \sum_{k=0}^{n-j} \binom{k + j - 2}{j - 2} X^{(j)} \psi_{n - j - k}.
\]
The recursion
\[
P_n = \left( \sum_{i=1}^N x_i \right) P_{n-1} + \sum_{j=2}^n (-1)^{j-1} \sum_{k=0}^{n-j} \binom{k + j - 2}{j - 2} \left( \sum_{i=1}^N x_i^j \right) P_{n - j - k}.
\]
for polynomials $P_n(x_1, \ldots, x_N, t)$ can be solved explicitly, but we find the resulting formula complicated and not particularly illuminating, and omit it from the article.
\end{Remark}

We can similarly upgrade various results proven in \cite{AnsFSM} for bounded free L{\'e}vy processes and uniform limits to general free compound Poisson processes and limits in probability. This applies to Theorem 1 (stochastic measures corresponding to crossing partitions are zero), Proposition 1 (for a centered process, stochastic measures corresponding to partitions with inner singletons are zero) and its corollary on the equality of expressions \eqref{Eq:Stochastic-integral} and \eqref{Eq:All-distinct},

\begin{Remark}
\label{Remark:Boxplus-power}
Let $\mu, \nu$ be probability measures on $\mf{R}$, such that $\mu = \mu_a$, $\nu =\mu_b$ for free $a, b \in (\tilde{\mc{A}}_{sa}, \tau)$. The additive free convolution $\mu \boxplus \nu$ is the distribution of $a + b$. If $\mu$ is supported on $\mf{R}_+$ (so that $a$ is positive), the multiplicative free convolution $\mu \boxtimes \nu$ is the distribution of $a^{1/2} b a^{1/2}$, which we identify (since $\tau$ is a trace) with the distribution of $a b$.

According to Proposition~3.5 in \cite{Belinschi-Nica-B_t}, we have the relation
\begin{equation}
\label{Eq:Boxplus-power}
(\mu^{\boxplus t}) \boxtimes (\nu^{\boxplus t}) = (\mu \boxtimes \nu)^{\boxplus t} \circ D_{1/t},
\end{equation}
where $D_{1/t}$ is the dilation operator corresponding to multiplying the operator by $t$. Note that in the proposition, the relation is stated for $t \geqslant 1$, but the same argument shows that it holds whenever all the convolution powers on the left-hand side are defined and at least one of them is supported on $\mf{R}_+$.
\end{Remark}

\begin{Prop}
\label{Prop:Dot-product}
Let
\[
\set{X_{i, N}^{(1)}: 1 \leqslant i \leqslant N, N \in \mf{N}} \cup \set{X_{i, N}^{(2)}: 1 \leqslant i \leqslant N, N \in \mf{N}} \subset \left((\tilde{\mc{A}}, \tau)_{sa} \right)
\]
be two triangular arrays with free, identically distributed rows, free from each other, the first of which consists of positive operators. Denote
\[
\sum_{i=1}^N X_{i, N}^{(j)} = X_N^{(j)}, \quad j = 1, 2
\]
and suppose that
\[
\lim_{N \rightarrow \infty}  X_N^{(j)} = X^{(j)}, \quad j = 1, 2
\]
in distribution, for some $\set{X^{(1)}, X^{(2)}}$. Then as $N \rightarrow \infty$,
\[
\sum_{i=1}^N \left(X_{i, N}^{(1)} \right)^{1/2} X_{i, N}^{(2)} \left( X_{i, N}^{(1)} \right)^{1/2} \rightarrow 0
\]
in distribution, and so also in probability.
\end{Prop}

\begin{proof}
Using the identity from the preceding remark,
\[
\mu_{\left(X_{i, N}^{(1)} \right)^{1/2} X_{i, N}^{(2)} \left( X_{i, N}^{(1)} \right)^{1/2}}
= (\mu_{X_N^{(1)}}^{\boxplus (1/N)}) \boxtimes (\mu_{X_N^{(2)}}^{\boxplus (1/N)})
= (\mu_{X_N^{(1)}} \boxtimes \mu_{X_N^{(2)}})^{\boxplus (1/N)} \circ D_N
\]
and so
\[
\mu_{\sum_{i=1}^N \left(X_{i, N}^{(1)} \right)^{1/2} X_{i, N}^{(2)} \left( X_{i, N}^{(1)} \right)^{1/2}}(z) = (\mu_{X_N^{(1)}} \boxtimes \mu_{X_N^{(2)}}) \circ D_N.
\]
As $N \rightarrow \infty$, $\mu_{X_N^{(1)}} \boxtimes \mu_{X_N^{(2)}} \rightarrow \mu_{X^{(1)}} \boxtimes \mu_{X^{(2)}}$ weakly, and so the distribution above converges to $\delta_0$ weakly.
\end{proof}

\begin{Remark}
Denote $\mc{C}_\mu(z) = z \phi_\mu(1/z)$ the free cumulant transform. A measure $\sigma$ is free regular if
\[
\mc{C}_\sigma(z) = \eta' z + \int_{\mf{R}} \left( \frac{1}{1 - z x} - 1 \right) \nu(dx)
\]
for some $\eta' \geqslant 0$ and $\nu((-\infty, 0]) = 0$. By Proposition~6.2 in \cite{Arizmendi-Hasebe-Sakuma-Subordinators}, if $\mu$ is $\boxplus$-infinitely divisible and symmetric, then
\[
\mu^2 = m \boxtimes \sigma.
\]
Here $\mu^2 = \mu^{(x^2)}$ in our earlier notation, $m$ is the standard free Poisson distribution, and $\sigma$ is a free regular measure. Moreover by Theorem~11 from \cite{Perez-Abreu-Sakuma-multiplicative-mixtures}, this is equivalent to
\[
\mc{C}_{\mu}(z) = \mc{C}_{\sigma}(z^2).
\]
Next, let $\mu, \nu$ be probability measures on $\mf{R}$, such that $\mu = \mu_a$, $\nu =\mu_b$ for free $a, b \in (\tilde{\mc{A}}_{sa}, \tau)$. Denote by $\mu \square \nu$ the distribution of the anti-commutator $a b + b a$. If $\mu, \nu$ are both symmetric, it coincides with the distribution of the commutator $i (a b - b a)$, and satisfies
\[
\left((\mu \square \nu)^{\boxplus {1/2}} \right)^2 = \mu^2 \boxtimes \nu^2.
\]
See \cite{Nica-Speicher-Commutators}, Lectures 15 and 19 in \cite{Nica-Speicher-book}, and Corollary~6.5 in \cite{Arizmendi-Hasebe-Sakuma-Subordinators}.

We also note that if in Remark~\ref{Remark:Boxplus-power}, $\mu$ is free regular, then by Theorem~4.2 in \cite{Arizmendi-Hasebe-Sakuma-Subordinators}, $\mu^{\boxplus t}$ is the distribution of a positive operator for all $t > 0$. So if in addition $\nu$ is $\boxplus$-infinitely divisible, the identity \eqref{Eq:Boxplus-power} holds for all such $t$.
\end{Remark}

\begin{Prop}
\label{Prop:Dot-product-2}
Let
\[
\set{X_{i, N}^{(1)}: 1 \leqslant i \leqslant N, N \in \mf{N}} \cup \set{X_{i, N}^{(2)}: 1 \leqslant i \leqslant N, N \in \mf{N}} \subset \left((\tilde{\mc{A}}, \tau)_{sa} \right)
\]
be two triangular arrays with free, identically distributed rows, free from each other, all of whose distributions are symmetric. Denote
\[
\sum_{i=1}^N X_{i, N}^{(j)} = X_N^{(j)}, \quad j = 1, 2
\]
and suppose that the distribution of each $X_N^{(j)}$ is $\boxplus$-infinitely divisible and
\[
\lim_{N \rightarrow \infty} X_N^{(j)} = X^{(j)}, \quad j = 1, 2
\]
in distribution, for some $\set{X^{(1)}, X^{(2)}}$. Then as $N \rightarrow \infty$,
\[
\sum_{i=1}^N \left(X_{i, N}^{(1)} X_{i, N}^{(2)} + X_{i, N}^{(2)} X_{i, N}^{(1)} \right) \rightarrow 0
\]
in distribution, and so also in probability, and
\begin{equation}
\label{Eq:Product}
\sum_{i=1}^N X_{i, N}^{(1)} X_{i, N}^{(2)} \rightarrow 0
\end{equation}
in probability.
\end{Prop}

\begin{proof}
Denote by $\mu_{j, N}$ the distribution of $X^{(j)}_N$. Using the preceding remark, we may write
\[
\mu_{j, N}^2 = m \boxtimes \sigma_{j, N},
\]
where $\sigma_{j, N}$ is a free regular measure, such that
\[
\mc{C}_{\mu_{j, N}}(z) = \mc{C}_{\sigma_{j, N}}(z^2).
\]
Note that
\[
\mc{C}_{\mu_{j, N}^{\boxplus (1/N)}}(z) = \frac{1}{N} \mc{C}_{\mu_{j, N}}(z) = \frac{1}{N} \mc{C}_{\sigma_{j, N}}(z^2) = \mc{C}_{\sigma_{j, N}^{\boxplus (1/N)}}(z^2).
\]
Thus
\[
\left( \mu_{j, N}^{\boxplus (1/N)} \right) ^2 = m \boxtimes \sigma_{j, N}^{\boxplus (1/N)}.
\]
Next,
\[
\begin{split}
\left( \left(\mu_{1, N}^{\boxplus (1/N)} \square \mu_{2, N}^{\boxplus (1/N)} \right)^{\boxplus (1/2)} \right)^2
& = \left( \mu_{1, N}^{\boxplus (1/N)} \right)^2 \boxtimes \left( \mu_{2, N}^{\boxplus (1/N)} \right)^2 \\
& = m \boxtimes \sigma_{1, N}^{\boxplus (1/N)} \boxtimes m \boxtimes \sigma_{2, N}^{\boxplus (1/N)}.
\end{split}
\]
Therefore
\[
\mc{C}_{\left(\mu_{1, N}^{\boxplus (1/N)} \square \mu_{2, N}^{\boxplus (1/N)} \right)^{\boxplus (N/2)}}(z)
= N \mc{C}_{\sigma_{1, N}^{\boxplus (1/N)} \boxtimes m \boxtimes \sigma_{2, N}^{\boxplus (1/N)}}(z^2).
\]
Applying the relation~\eqref{Eq:Boxplus-power} twice and distributing the dilation, we get
\[
\begin{split}
\left( \sigma_{1, N}^{\boxplus (1/N)} \boxtimes m \boxtimes \sigma_{2, N}^{\boxplus (1/N)} \right)^{\boxplus N}
& = \left( \sigma_{1, N} \boxtimes m^{\boxplus N} \boxtimes \sigma_{2, N} \right) \circ D_{N^2} \\
& = \left( m^{\boxplus N} \circ D_{N} \right) \boxtimes \left( (\sigma_{1, N} \boxtimes \sigma_{2, N}) \circ D_{N} \right).
\end{split}
\]
Using the (noncommutative) law of large numbers, or by a direct calculation, $m^{\boxplus N} \circ D_{N} \rightarrow \delta_1$, so these measures converge to $\delta_0$ weakly. Therefore their free cumulant transforms converge to zero pointwise, which implies that
\[
\left(\mu_{1, N}^{\boxplus (1/N)} \square \mu_{2, N}^{\boxplus (1/N)} \right)^{\boxplus (N/2)} \rightarrow \delta_0.
\]

Since the same convergence in probability holds for the commutators
\[
\sum_{j=1}^N i \left(X_{j, N}^{(1)} X_{j, N}^{(2)} - X_{j, N}^{(2)} X_{j, N}^{(1)} \right),
\]
it holds for their linear combination \eqref{Eq:Product}.
\end{proof}

\begin{proof}[Proof of Theorem~\ref{Thm:Quadratic-Cubic}]
Let
\begin{multline*}
X(t)=\eta t \mathbf{1}_{\mathcal{A}^{0}}+ \sqrt{a} S(t) \\
+\lim_{\epsilon\searrow 0} \Bigl(\int_{(0,t]\times \set{|x|>\epsilon}}xdM(t,x)
-\int_{(0,t]\times \set{\epsilon<|x|\leqslant 1}}x ((\mathrm{Leb} \otimes \rho)(dt,dx)\mathbf{1}_{\mathcal{A}^{0}}) \Bigr).
\end{multline*}
Fix $\alpha \in (0, 1)$. Denote
\begin{multline*}
X'(t)= \left( \eta - \int_{\set{\alpha \leqslant |x|\leqslant 1}}x \,\rho(dx) \right) t \mathbf{1}_{\mathcal{A}^{0}}+ \sqrt{a} S(t) \\
+ \lim_{\epsilon\searrow 0} \Bigl(\int_{(0,t]\times \set{\eps < |x| < \alpha}}xdM(t,x)
-\int_{(0,t]\times \set{\epsilon<|x| < \alpha}}x ((\mathrm{Leb} \otimes \rho)(dt,dx)\mathbf{1}_{\mathcal{A}^{0}}) \Bigr).
\end{multline*}
and
\[
X''(t)=
\int_{(0,t]\times \set{|x| \geqslant \alpha}}xdM(t,x).
\]
Note that $\set{X''(t)}$ is an (unbounded) free compound Poisson process, $X(t) = X'(t) + X''(t)$, $\set{X'(t)}$ and $\set{X''(t)}$ are free from each other, and all of their distributions are $\boxplus$-infinitely divisible and symmetric. Then
\[
\sum_{i=1}^{[N t]} X_{i, N}^2 = \sum_{i=1}^{[N t]} \left(X_{i, N}' \right)^2 + \sum_{i=1}^{[N t]} \left(X_{i, N}'' \right)^2 + \sum_{i=1}^{[N t]} \left(X_{i, N}' X_{i, N}'' + X_{i, N}'' X_{i, N}' \right).
\]
By Theorem~\ref{Thm:Compound-Poisson}, the second term converges to $\left( X'' \right)^{(2)}(t)$ in probability. By Proposition~\ref{Prop:Dot-product-2}, the third term converges to zero in probability. By Theorem~\ref{Thm:Variations-distribution}, for fixed $\alpha$, the first term converges in distribution to
\[
\left( X' \right)^{(2)}(t) = a t \textbf{1}_{\mathcal{A}}+\int_{(0,t]\times (-\alpha, \alpha)}x^2 dM(t,x).
\]
Finally, as $\alpha \rightarrow 0$, $\left( X' \right)^{(2)}(t) \rightarrow a t \textbf{1}_{\mathcal{A}}$ in probability. Thus, given $\eps, \delta > 0$, we may choose $\alpha$ small so that $\left( X' \right)^{(2)}(t) - a t \textbf{1}_{\mathcal{A}} \in \mc{N}(\eps, \delta)$. Then for sufficiently large $N$, $\sum_{i=1}^{[N t]} \left(X_{i, N}' \right)^2 - a t \textbf{1}_{\mathcal{A}} \in \mc{N}(\eps, \delta)$ and
\[
\sum_{i=1}^{[N t]} X_{i, N}^2 - \left( X'' \right)^{(2)}(t) - a t \textbf{1}_{\mathcal{A}} \in \mc{N}(\eps, \delta).
\]
It remains to note that also
\[
X^{(2)}(t) - \left( X'' \right)^{(2)}(t) - a t \textbf{1}_{\mathcal{A}}
= \left( X' \right)^{(2)}(t) - a t \textbf{1}_{\mathcal{A}} \in \mc{N}(\eps, \delta). \qedhere
\]
\end{proof}

We finish this section with another possible definition of joint convergence in distribution. As already noted, for commuting variables, convergence in distribution of linear combinations is equivalent to joint convergence in distribution. As pointed out by {\'E}duard Maurel-Segala and Maxime Fevrier, this is not the case for non-commuting variables. However the following matricial version is its natural replacement. By the well-known linearization trick \cite{Haa-Thor-Ext} (see also Chapter 10 of \cite{Mingo-Speicher-book}), it implied the definition in the introduction; we do not know if they are in general equivalent. We show that convergence in probability implies joint convergence in this possibly stronger sense as well.

\begin{Defn}
\label{Defn:Joint-2}
Let
\[
\set{x_{i, N}: 1 \leqslant i \leqslant k, N \in \mf{N}} \cup \set{x_i : 1 \leqslant i \leqslant k} \subset (\tilde{\mc{A}}_{sa}, \tau).
\]
We say that $(x_{1, N}, \ldots, x_{k, N}) \rightarrow (x_1, \ldots x_k)$ jointly in distribution if for any $d$ and any Hermitian matrices $A_1, \ldots, A_k \in M_d(\mf{C})$, and any $B \in M_d(\mf{C})$ with $\Im B >  \eps I$ for some $\eps > 0$, the Cauchy transforms
\[
(I \otimes \tau) \left( B \otimes 1 - \sum_{i=1}^k A_i \otimes x_{i, N} \right)^{-1} \rightarrow (I \otimes \tau) \left( B \otimes 1 - \sum_{i=1}^k A_i \otimes x_{i} \right)^{-1}
\]
in norm in $M_N(\mf{C})$.
\end{Defn}

\begin{Prop}
If for each $i$, $x_{i, N} \rightarrow x_i$ in probability, then $(x_{1, N}, \ldots, x_{k, N}) \rightarrow (x_1, \ldots x_k)$ in the sense of Definition~\ref{Defn:Joint-2}.
\end{Prop}

\begin{proof}
The argument in Proposition~2.19 in \cite{BNTSelf} largely goes through; we outline it for the reader's convenience. Note first that for $X \in M_d(\tilde{\mc{A}}_{sa})$,
\[
\norm{(B \otimes 1 - X)^{-1}} \leqslant \norm{(\Im B)^{-1}},
\]
and in particular this operator is bounded. By the resolvent identity,
{\tiny
\begin{multline*}
\left( B \otimes 1 - \sum_{i=1}^k A_i \otimes x_{i, N} \right)^{-1} - \left( B \otimes 1 - \sum_{i=1}^k A_i \otimes x_{i} \right)^{-1} \\
= \left( B \otimes 1 - \sum_{i=1}^k A_i \otimes x_{i, N} \right)^{-1} \left(\sum_{i=1}^k A_i \otimes x_{i} - \sum_{i=1}^k A_i \otimes x_{i, N} \right) \left( B \otimes 1 - \sum_{i=1}^k A_i \otimes x_{i} \right)^{-1}.
\end{multline*}
}
By assumption and a short argument, for any $\eps, \delta > 0$ there is an $n$ such that for $N \geqslant n$, there is a projection $p_N$ with $\tau[p_N] > 1 - \delta$ and
\[
\norm{\left(\sum_{i=1}^k A_i \otimes x_{i} - \sum_{i=1}^k A_i \otimes x_{i, N} \right) (I \otimes p_N)} < \eps \sum_{i=1}^k \norm{A_i}.
\]
Thus for some projection $q_N$ with the same property,
\begin{multline*}
\norm{\left(\left( B \otimes 1 - \sum_{i=1}^k A_i \otimes x_{i, N} \right)^{-1} - \left( B \otimes 1 - \sum_{i=1}^k A_i \otimes x_{i} \right)^{-1} \right) (I \otimes q_N)} \\
\leqslant \eps \norm{(\Im B)^{-1}}^2 \sum_{i=1}^k \norm{A_i}.
\end{multline*}
In particular, the same estimate holds on each matrix entry on the left-hand side. Applying the rest of the argument from Proposition~2.19 in \cite{BNTSelf} entry-wise, it follows that
\[
(I \otimes \tau) \left[ \left( B \otimes 1 - \sum_{i=1}^k A_i \otimes x_{i, N} \right)^{-1} - \left( B \otimes 1 - \sum_{i=1}^k A_i \otimes x_{i} \right)^{-1} \right] \rightarrow 0.
\]
\end{proof}

\appendix

\section{Symmetric polynomials in non-commuting variables}

Symmetric functions in non-commuting variables (not to be confused with non-commutative symmetric functions) have been considered in \cite{Rosas-Sagan-Symmetric-nc,Bergeron-Rosas-Invariants-nc} and subsequent work. We need the following observation, whose explicit statement we could not find in the literature.

\begin{Prop}
Let $p_k = \sum_{i=1}^N x_i^k$ be the basic power sum symmetric polynomials. In the algebra of non-commutative polynomials $\mf{C} \langle x_1, \ldots, x_N \rangle$, the subalgebra generated by $\set{p_k: k \geqslant 1}$ is the linear span of polynomials
\[
P_{\mb{u}}(\mb{x}) = \sum_{\substack{i(1), i(2), \ldots = 1 \\ \text{neighbors distinct}}}^N x_{i(1)}^{u(1)} x_{i(2)}^{u(2)} \ldots
\]
for all choices of $\mb{u}$ with coordinates $u(i) \geqslant 1$. Note that these polynomials are obviously linearly independent. In particular, the elementary symmetric functions
\[
e_k = \sum_{\substack{i(1) \neq i(2) \neq \ldots \neq i(r) \\ i(1) + i(2)+ \ldots + i(r) = k}} x_{i(1)} x_{i(2)} \ldots x_{i(r)}
\]
are not in this subalgebra for $k > 1$.
\end{Prop}

\begin{proof}
Clearly the algebra generated by all $p_k$ is the span of all
\[
Q_{\mb{u}}(\mb{x}) = p_{u(1)}(\mb{x}) p_{u(2)}(\mb{x}) \ldots
= \sum_{i(1), i(2), \ldots = 1}^N x_{i(1)}^{u(1)} x_{i(2)}^{u(2)} \ldots,
\]
where the $i(j)$ are not necessarily distinct. Denote by $\Int(n)$ the interval partitions of $[n]$. Then we may re-index these polynomials as
\[
Q_{\pi}(\mb{x}) = \sum_{i(1), i(2), \ldots, i(r) = 1}^N \prod_{j=1}^r x_{i(j)}^{\abs{V_j}}
= \prod_{j=1}^r p_{\abs{V_j}}(\mb{x})
\]
for $\pi = \set{V_1, \ldots, V_r} \in \Int(n)$ for some $n$. For $\mb{u} \in [N]^r$, denote $\ker(\mb{u}) \in \Part(n)$ the partition such that $u(i) = u(j)$ if and only if $i, j$ lie in the same block of $\ker(\mb{u})$. Note that for $V \in \ker(\mb{u})$, the notation $u(V)$ is unambiguous. Also, for $\pi \in \Part(n)$, let $I(\pi)$ be the largest interval partition such that $I(\pi) \leqslant \pi$. Note that $I(\pi) = \tau$ if $\pi \geqslant \tau$ and if $V, V'$ are neighboring blocks of $\tau$, they lie in different blocks of $\pi$. Finally, for $\pi = \set{V_1, \ldots, V_r} \in \Int(n)$, denote
\[
P_\pi(\mb{x}) = \sum_{\substack{i(1), i(2), \ldots, i(r) = 1 \\ \text{neighbors distinct}}}^N \prod_{j=1}^r x_{i(j)}^{\abs{V_j}}.
\]
Then for $\sigma \in \Int(n)$,
\[
\begin{split}
Q_\sigma(\mb{x})
& = \sum_{\substack{\pi \in \Part(n) \\ \pi \geqslant \sigma}} \sum_{\mb{i} : \ker{\mb{i}} = \pi} \prod_{V \in \pi} x_{i(V)}^{\abs{V}}
= \sum_{\substack{\tau \in \Int(n) \\ \tau \geqslant \sigma}} \sum_{\substack{\pi \in \Part(n) \\ I(\pi) = \tau}} \sum_{\mb{i} : \ker{\mb{i}} = \pi} \prod_{V \in \pi} x_{i(V)}^{\abs{V}} \\
& = \sum_{\substack{\tau \in \Int(n) \\ \tau \geqslant \sigma}} \sum_{\substack{i(1), i(2), \ldots, i(\abs{\tau}) = 1 \\ \text{neighbors distinct}}}^N \prod_{V \in \tau} x_{i(V)}^{\abs{V}}
= \sum_{\substack{\tau \in \Int(n) \\ \tau \geqslant \sigma}} P_\tau(\mb{x}).
\end{split}
\]
Then by M{\"o}bius inversion on the lattice $\Int(n)$, the spans of $\set{Q_\pi}$ and of $\set{P_\pi}$ are the same.
\end{proof}

\begin{Cor}
\label{Cor:Distinct}
In the notation of the preceding proof,
\[
P_\sigma = \sum_{\substack{\pi \in \Int(n) \\ \pi \geqslant \sigma}} (-1)^{\abs{\sigma} - \abs{\pi}} \prod_{V \in \pi} p_{\abs{V}}(\mb{x}).
\]
In particular,
\[
\sum_{\substack{i(1), i(2), \ldots, i(n) = 1 \\ \text{neighbors distinct}}}^N \prod_{j=1}^n x_{i(j)}
= \sum_{\pi \in \Int(n)} (-1)^{n - \abs{\pi}} \prod_{V \in \pi} p_{\abs{V}}(\mb{x}).
\]
\end{Cor}

\begin{proof}
The first statement follows by M{\"o}bius inversion, since the M{\"o}bius function on the lattice $\Int(n)$ is $\mathrm{M\ddot{o}b}(\sigma, \pi) = (-1)^{\abs{\sigma} - \abs{\pi}}$. The second statement follows from the fact that the left-hand side is $P_{\hat{0}_n}(\mb{x})$.
\end{proof}

\normalsize

\def\cprime{$'$} \def\cprime{$'$}
\providecommand{\bysame}{\leavevmode\hbox to3em{\hrulefill}\thinspace}
\providecommand{\MR}{\relax\ifhmode\unskip\space\fi MR }
\providecommand{\MRhref}[2]{%
  \href{http://www.ams.org/mathscinet-getitem?mr=#1}{#2}
}
\providecommand{\href}[2]{#2}

\end{document}